\documentclass[a4paper,10pt]{amsart}
\usepackage{graphicx}
\usepackage{subfigure}
\usepackage[utf8]{inputenc}
\usepackage{amsmath,amssymb}
\usepackage{verbatim}
\usepackage[colorlinks=true]{hyperref}
\usepackage[usenames,dvipsnames]{color}
\usepackage{xspace}

\usepackage{tikz}
\usetikzlibrary{snakes,trees,calc}

\newtheorem{thm}{Theorem}[section]
\newtheorem{prop}[thm]{Proposition}
\newtheorem{lemma}[thm]{Lemma}
\newtheorem{cor}[thm]{Corollary}
\newtheorem{conj}[thm]{Conjecture}

\theoremstyle{definition}
\newtheorem{defi}[thm]{Definition}

\newcommand{\ie}{\emph{i.e.}\xspace}
\newcommand{\Ie}{\emph{I.e.}\xspace}

\newcommand{\si}{\sigma}
\newcommand{\C}{C^{-1}}

\newcommand{\ai}{a}
\newcommand{\aj}{b}
\newcommand{\ak}{c}
\newcommand{\at}{{b\bullet c}}
\newcommand{\bi}{\alpha}
\newcommand{\bj}{\beta}
\newcommand{\bk}{\gamma}
\newcommand{\bt}{{\beta\circ \gamma}}
\newcommand{\zh}{{\color{orange} h}}
\newcommand{\zi}{{\color{red} i}}
\newcommand{\zj}{j}
\newcommand{\zk}{{\color{blue}k}}
\newcommand{\zl}{{\color{green}l}}
\newcommand{\wh}{{\color{orange}\delta}}
\newcommand{\wi}{{\color{red} \epsilon}}
\newcommand{\wj}{\zeta}
\newcommand{\wk}{{\color{blue}\eta}}
\newcommand{\wl}{{\color{green}\theta}}
\newcommand{\q}[2]{q_{#1 #2}}
\newcommand{\um}[2]{1_{#1 #2}}

\title[Polynomials counting FPL configurations at negative values]{On some polynomials enumerating Fully Packed Loop configurations, evaluation at negative values}
\author{Tiago Fonseca}
\address{Centre de Recherches Mathématiques, Université de Montréal}
\email{tiago.fonseca@lapth.cnrs.fr}
\begin{document}

\tikzstyle{ASM}=[black]
\tikzstyle{arche} = [red, thick]
\tikzstyle{line} = [black, semithick]
\tikzstyle{ajuda} = [gray, very thin]
\tikzstyle{dyck} = [black, thick]
\tikzstyle{young} = [black, semithick]
\tikzstyle{FPL}=[very thick,blue,rounded corners=3pt]
\tikzstyle{zFace}=[fill=yellow!50,draw=white,very thin]
\tikzstyle{xFace}=[fill=orange!30,draw=white,very thin]
\tikzstyle{yFace}=[fill=blue!30,draw=white,very thin]
\tikzstyle{nilp}=[blue,thick]
\tikzstyle{nilP}=[red,thick,dotted]
\tikzstyle{Nilp}=[blue,inner sep=0pt,minimum size=2pt,fill,shape=circle]
\tikzstyle{NilP}=[red,inner sep=0pt,minimum size=2pt,fill,shape=circle]

\begin{abstract}

In this article, we are interested in the enumeration of Fully Packed Loop configurations on a grid with a given noncrossing matching. 
These quantities also appear as the groundstate components of the $O(n)$ Loop model as conjectured by Razumov and Stroganov and recently proved by Cantini and Sportiello.

When considering matchings with $p$ nested arches these quantities are known to be polynomials. 
In a recent article, Fonseca and Nadeau conjectured some unexpected properties of these polynomials, suggesting that these quantities could be combinatorially interpreted even for negative $p$.
Here, we prove some conjectures in this article.
Notably, we prove that for negative $p$ we can factor the polynomials into two parts a ``positive'' one and a ``negative'' one. 
Also, a sum rule of the negative part is proven here. 

\end{abstract}

\maketitle


 \section*{Introduction}

In 2001, Razumov and Stroganov~\cite{RS-conj} conjectured that there is a correspondence between the Fully Packed Loop (FPL) configurations, a combinatorial model, and the components of the groundstate vector in the $O(n)$ Loop model, a model in statistical physics.
On the one hand, the connectivity of the FPL configurations at the boundary is described by a perfect noncrossing matchings $\pi$ of $2n$ points (see definitions~\ref{representations}).
The number of FPL configurations associated to a certain matching $\pi$ is denoted by $A_\pi$.
On the other hand, the $O(n)$ model is defined on the set of matchings and the groundstate components are naturally indexed by the matchings and are written $\psi_\pi$.
Razumov and Stroganov conjectured that this quantities are the same $A_\pi = \psi_\pi$ for all matchings $\pi$.
This conjecture was proved in 2010 by Cantini and Sportiello~\cite{ProofRS}.

Consider matchings with $p$ nested arches surrounding a smaller matching $\pi$, which we denote $(\pi)_p = (\cdots(\pi)\cdots)$. 
It was conjectured in~\cite{Zuber-conj}, and after proved in~\cite{CKLN,artic47}, that the quantities $A_{(\pi)_p}$ and $\psi_{(\pi)_p}$ are polynomials in $p$.
In a recent article, Nadeau and Fonseca~\cite{negative} conjectured some surprising properties of these polynomials.
The goal of this article is to prove some of these conjectures, notably 3.8 and 3.11.

\medskip

Let $\pi$ be a matching composed by $n$ arches.
We denote by $A_\pi(p)$ (respectively $\psi_\pi(p)$) the polynomial which coincides with $A_{(\pi)_p}$ (respectively $\psi_{(\pi)_p}$) when $p$ is a nonnegative integer.
In this paper we prove that, for $p$ between $0$ and $-n$, these quantities are either zero or they can be seen as the product of two distinct terms. 
One being a new quantity $g_\pi$ also indexed by perfect noncrossing matchings and the other being again the quantities $A_\pi$.

The quantities $g_\pi$ are surprisingly connected with the Fully Packed Loop model: the sum of the absolute values of $g_\pi$ is equal to the number of FPL configurations.
This relation has been proven in~\cite{tese}.
Here we prove another sum rule also conjectured in~\cite{negative}: the sum of the quantities $g_\pi$ is equal to the number of vertically symmetric FPL configurations, up to an eventual sign. 
These sum rules, together with other properties of $g_\pi$, rises the idea that these numbers $g_\pi$ have some combinatorial meaning, \ie they are counting something which is related to the FPL configurations.

An interesting byproduct of the proofs are the multivariate integral formulæ proposed for $g_\pi$, which allow us to reformulate the first mentioned conjecture in a stronger form (a polynomial form).

 \medskip

Let us give a detailed outline of this article. 
In the first section, we introduce the two models: the Fully Packed Loop (FPL) model and the $O(n)$ Loop model, and the associated quantities $A_\pi$ and $\psi_\pi$, respectively.
Furthermore, we give a brief perspective of the case of $\pi$ when it contains $p$ nested arches.

In Section~\ref{sec:conj} we state the two conjectures that we solve here.
They concern the polynomials $A_\pi (t)$.

In order to prove the first one, we introduce a multivariate polynomial version of the quantities $\psi_\pi$ in Section~\ref{sec:CPL_multi}, defined though the quantum Knizhnik--Zamolodchikov equation.
Although it seems a more complicated approach, this version allows us to use some polynomial properties, which will be essential to the proof.

The two further sections are dedicated to the proof of the conjectures.
The paper finishes with an appendix, where we describe some important results, which are straightforward but a little bit tedious.


\section{Definitions}\label{sec:defi}

In this section we introduce the concept of matchings. 
Furthermore, we briefly describe the Fully Packed Loop model and the $O(n)$ Loop model.
Finally, we introduce the concept of nested matching.

\subsection{Matchings}\label{representations}
A matching\footnote{these matchings are usually called {\em perfect noncrossing matchings} in the literature, but this is the only kind of matchings we will encounter so there will be no possible confusion.} $\pi$ of size $n$ is defined as a set of $n$ disjoint pairs of integers $\{1,\ldots,2n\}$, which are {\em noncrossing} in the sense that if $\{i,j\}$ and $\{k,l\}$ are two pairs in $\pi$ with $i<j$ and $k<l$, then it is forbidden to have $i<k<j<l$ or $k<i<l<j$. 
The number of matchings with $n$ pairs is the Catalan number $c_n=\frac{1}{n+1}\binom{2n}{n}$.

Matchings can be represented in several ways:

\begin{itemize}
 \item A Link Pattern is a set of noncrossing arches on $2n$ horizontally aligned points labelled from $1$ to $2n$. 
       Given a pair in a matching $\{i,j\}$, the corresponding arch connects point $i$ to the point $j$. 
       This will be our standard representation;
\[
\{\{1,2\},\{3,6\},\{4,5\}\}
\Leftrightarrow
 \begin{tikzpicture}[scale=0.25]
  \draw[arche] (0,0) .. controls (0,.5) and (1,.5) .. (1,0);
  \draw[arche] (2,0) .. controls (2,1.5) and (5,1.5) .. (5,0); 
  \draw[arche] (3,0) .. controls (3,.5) and (4,.5) .. (4,0);
  \draw[line] (-.5,0) -- (5.5,0);
 \end{tikzpicture}
\]

 \item A well-formed sequence of parentheses, also called \emph{parenthesis word}. 
Given an arch in a matching, the point connected to the left (respectively to the right) is encoded by an opening parenthesis (resp. by a closing parenthesis);
\[
 \begin{tikzpicture}[scale=0.25]
  \draw[arche] (0,0) .. controls (0,.5) and (1,.5) .. (1,0);
  \draw[arche] (2,0) .. controls (2,1.5) and (5,1.5) .. (5,0); 
  \draw[arche] (3,0) .. controls (3,.5) and (4,.5) .. (4,0);
  \draw[line] (-.5,0) -- (5.5,0);
 \end{tikzpicture}
\Leftrightarrow ()(())
\]

 \item A Dyck Path, which is a path between $(0,0)$ and $(2n,0)$ with steps NE $(1,1)$ and SE $(1,-1)$ that never goes under the horizontal line $y=0$. 
An opening parenthesis corresponds to a NE step, and a closing one to a SE step;
\[
 ()(()) \Leftrightarrow
 \begin{tikzpicture}[scale=0.25, baseline=2pt]
  \draw[dyck] (0,0) -- (1,1) -- (2,0) -- (3,1) -- (4,2) -- (5,1) -- (6,0);
 \end{tikzpicture}
\]

 \item A Young diagram is a collection of boxes, arranged in left-justified rows, such that the size of the rows is weakly decreasing from top to bottom.
Matchings with $n$ arches are in bijection with Young diagrams such that the $i$th row from the top has no more than $n-i$ boxes. 
The Young diagram can be constructed as the complement of a Dyck path, rotated $45^\circ$ counterclockwise;
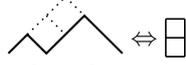
\[
 \begin{tikzpicture}[scale=0.25, baseline=3pt]
   \draw[dyck] (0,0) -- (1,1) -- (2,0) -- (3,1) -- (4,2) -- (5,1) -- (6,0);
   \draw[young, dotted] (1,1) -- (3,3);
   \draw[young, dotted] (2,0) -- (4,2);
   \draw[young, dotted] (1,1) -- (2,0);
   \draw[young, dotted] (2,2) -- (3,1);
   \draw[young, dotted] (3,3) -- (4,2);
 \end{tikzpicture}
\Leftrightarrow
 \begin{tikzpicture}[scale=0.25, baseline=-10pt]
   \draw[young] (0,0) -- (0,-2);
   \draw[young] (1,0) -- (1,-2);
   \draw[young] (0,0) -- (1,0);
   \draw[young] (0,-1) -- (1,-1);
   \draw[young] (0,-2) -- (1,-2);
 \end{tikzpicture}
\]

\item A sequence $a=\{a_1,\ldots,a_n\}\subseteq\{1,\ldots,2n\}$, such that $a_{i-1}<a_i$ and $a_i\leq 2i-1$ for all $i$. 
Here $a_i$ is the position of the $i$th opening parenthesis.
\[
 ()(()) \Leftrightarrow \{1,3,4\}
\]
\end{itemize}

We will often identify matchings under those different representations, through the bijections explained above. 
We may need at times to stress a particular representation: thus we write $Y(\pi)$ for the Young diagram associated to $\pi$, and $a(\pi)$ for the increasing sequence associated to $\pi$, etc...

We will represent $p$ nested arches around a matching $\pi$  by ``$(\pi)_p$'', and $p$ consecutive small arches by ``$()^p$''; thus for instance 
\[
((((()()))))()()()=(()^2)_4()^3.
\]

We define a {\em partial order} on matchings as follows: $\si \preceq \pi$ if the Young diagram of $\pi$ contains the Young diagram of $\si$, that is $Y(\si)\subseteq Y(\pi)$. 
In the Dyck path representation, this means that the path corresponding to $\si$ is always weakly above the path corresponding to $\pi$; in the sequence representation, if we write $a=a(\si)$ and $a'=a(\pi)$, then this is simply expressed by $a_i\leq a'_i$ for all $i$.

Given a matching $\pi$, we define $d(\pi)$ as the total number of boxes in the Young diagram $Y(\pi)$. 
We also let $\pi^*$ be the conjugate matching of $\pi$, defined by: $\{i,j\}$ is an arch in $\pi^*$ if and only if $\{2n+1-j,2n+1-i\}$ is an arch in $\pi$.
This corresponds to a mirror symmetry of the parenthesis word, and a transposition in the Young diagram. 
We also define a natural {\em rotation} $r$ on matchings: $i,j$ are linked by an arch in $r(\pi)$ if and only if $i+1,j+1$ are linked in $\pi$ (where indices are taken modulo $2n$). 
These last two notions are illustrated on Figure~\ref{fig:matchings}.

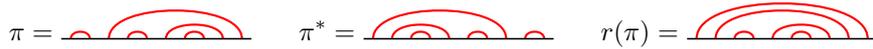
\begin{figure}[!ht]
\begin{align*}
\pi&=
 \begin{tikzpicture}[scale=0.25]
  \draw[arche] (0,0) .. controls (0,.5) and (1,.5) .. (1,0);
  \draw[arche] (2,0) .. controls (2,2) and (9,2) .. (9,0); 
  \draw[arche] (3,0) .. controls (3,.5) and (4,.5) .. (4,0);
  \draw[arche] (5,0) .. controls (5,1) and (8,1) .. (8,0);
  \draw[arche] (6,0) .. controls (6,.5) and (7,.5) .. (7,0);
  \draw[line] (-.5,0) -- (9.5,0);
 \end{tikzpicture}
 &
 \pi^*&=
 \begin{tikzpicture}[scale=0.25]
  \draw[arche] (9,0) .. controls (9,.5) and (8,.5) .. (8,0);
  \draw[arche] (7,0) .. controls (7,2) and (0,2) .. (0,0); 
  \draw[arche] (6,0) .. controls (6,.5) and (5,.5) .. (5,0);
  \draw[arche] (4,0) .. controls (4,1) and (1,1) .. (1,0);
  \draw[arche] (3,0) .. controls (3,.5) and (2,.5) .. (2,0);
  \draw[line] (-.5,0) -- (9.5,0);
 \end{tikzpicture}
 &
 r(\pi)&=
 \begin{tikzpicture}[scale=0.25]
  \draw[arche] (0,0) .. controls (0,2.5) and (9,2.5) .. (9,0);
  \draw[arche] (1,0) .. controls (1,2) and (8,2) .. (8,0); 
  \draw[arche] (2,0) .. controls (2,.5) and (3,.5) .. (3,0);
  \draw[arche] (4,0) .. controls (4,1) and (7,1) .. (7,0);
  \draw[arche] (5,0) .. controls (5,.5) and (6,.5) .. (6,0);
  \draw[line] (-.5,0) -- (9.5,0);
 \end{tikzpicture}
\end{align*}
\caption{A matching, its conjugate, and the rotated matching.\label{fig:matchings}}
\end{figure}

We need additional notions related to the Young diagram representation. 
So let $Y$ be a young diagram, and $u$ one of its boxes. 
The {\em hook length} $h(u)$ is the number of boxes below $u$ in the same column, or to its right in the same row (including the box $u$ itself). 
We note $H_Y$ the product of all hook lengths, \ie $H_Y=\prod_{u\in Y} h(u)$. 

\subsection{Fully Packed Loop}\label{sub:FPLintro}

 A \emph{Fully Packed Loop configuration} (FPL) of size $n$ is a subgraph of the square grid with $n^2$ vertices, such that each vertex is connected to exactly two edges. 
We furthermore impose the following boundary conditions: we select alternatively every second of the external edges to be part of our FPLs. 
By convention, we fix that the leftmost external edge on the top boundary is part of the selected edges, which fixes thus the entire boundary of our FPLs. 
We number these external edges clockwise from $1$ to $2n$, see Figure~\ref{fig:fplexample}.
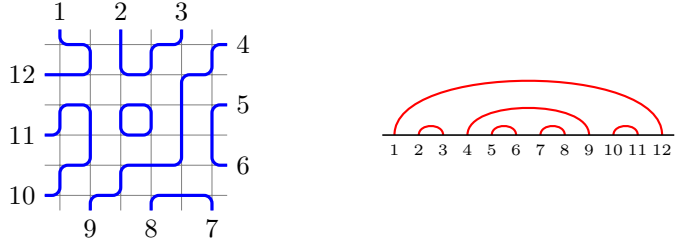
\begin{figure}[!ht]
 \begin{center}
  \begin{tikzpicture}[scale=.4]
  \draw[step=1,ajuda] (-.5,-.5) grid (5.5,5.5);
  \draw[FPL] (-.5,0) -- (0,0) -- (0,1) -- (1,1) -- (1,2) -- (1,3) -- (0,3) -- (0,2) -- (-.5,2);
  \draw[FPL] (1,-.5) -- (1,0) -- (2,0) -- (2,1) -- (3,1) -- (4,1) -- (4,2) -- (4,3) -- (4,4) -- (5,4) -- (5,5) -- (5.5,5);
  \draw[FPL] (3,-.5) -- (3,0) -- (4,0) -- (5,0) -- (5,-.5);
  \draw[FPL] (5.5,1) -- (5,1) -- (5,2) -- (5,3) -- (5.5,3);
  \draw[FPL] (-.5,4) -- (0,4) -- (1,4) -- (1,5) -- (0,5) -- (0,5.5);
  \draw[FPL] (2,5.5) -- (2,5) -- (2,4) -- (3,4) -- (3,5) -- (4,5) -- (4,5.5); 
  \draw[FPL] (2,2) rectangle (3,3);
  \node[above] at (0,5.5) {1};
  \node[above] at (2,5.5) {2};
  \node[above] at (4,5.5) {3};
  \node[right] at (5.5,5) {4};
  \node[right] at (5.5,3) {5};
  \node[right] at (5.5,1) {6};
  \node[below] at (5,-.5) {7};
  \node[below] at (3,-.5) {8};
  \node[below] at (1,-.5) {9};
  \node[left] at (-.5,0) {10};
  \node[left] at (-.5,2) {11};
  \node[left] at (-.5,4) {12};

 \begin{scope}[shift={(11,2)},scale=.8]
  \draw[arche] (0,0) .. controls (0,3) and (11,3) .. (11,0);
  \draw[arche] (1,0) .. controls (1,.5) and (2,.5) .. (2,0); 
  \draw[arche] (3,0) .. controls (3,1.5) and (8,1.5) .. (8,0);
  \draw[arche] (4,0) .. controls (4,.5) and (5,.5) .. (5,0);
  \draw[arche] (6,0) .. controls (6,.5) and (7,.5) .. (7,0);
  \draw[arche] (9,0) .. controls (9,.5) and (10,.5) .. (10,0);
  \draw[line] (-.5,0) -- (11.5,0);
  \node[below] at (0,0) {\tiny{1}};
  \node[below] at (1,0) {\tiny{2}};
  \node[below] at (2,0) {\tiny{3}};
  \node[below] at (3,0) {\tiny{4}};
  \node[below] at (4,0) {\tiny{5}};
  \node[below] at (5,0) {\tiny{6}};
  \node[below] at (6,0) {\tiny{7}};
  \node[below] at (7,0) {\tiny{8}};
  \node[below] at (8,0) {\tiny{9}};
  \node[below] at (9,0) {\tiny{10}};
  \node[below] at (10,0) {\tiny{11}};
  \node[below] at (11,0) {\tiny{12}};
  \end{scope}
  \end{tikzpicture}
\end{center}
\caption{FPL with its associated matching \label{fig:fplexample}}
\end{figure}

In each FPL configuration $F$ the chosen external edges are clearly linked by paths which  do not cross each other. 
We define $\pi(F)$ as the set of pairs $\{i,j\}$ of integers in $\{1,\ldots,2n\}$ such that the external edges labeled $i$ and $j$ are linked by a path in $F$. 
Then $\pi(F)$ is a matching in the sense of Section~\ref{representations}; an example is given on the right of Figure~\ref{fig:fplexample}. 

\begin{defi}[$A_\pi$]
 For any matching $\pi$, we define $A_\pi$ as the number of FPLs $F$ such that $\pi(F)=\pi$.
\end{defi}

A result of Wieland~\cite{wieland} shows that a rotation on matchings leaves the numbers $A_\pi$ invariant, and it is then easily seen that conjugation of matchings also leaves them invariant:

\begin{thm}[\cite{wieland}]
\label{thm:invar_api}
For any matching $\pi$, we have $A_\pi=A_{r(\pi)}$ and $A_\pi=A_{\pi^*}$.
\end{thm}

 Now we let $A_n$ be the total number of FPLs of size $n$; by definition we have $A_n=\sum_\pi A_\pi$ where $\pi$ goes through all matchings with $n$ arches. 
We also define $A_{n}^V$ as the number of FPLs of size $n$ which are invariant with respect to vertical symmetry. 
It is easily seen that $A_{2n}^V=0$.
We have the famous product expressions of these quantities:
\begin{align}
 A_n&=\prod_{k=0}^{n-1} \frac{(3k+1)!}{(n+k)!}; \\
A_{2n+1}^V&= \frac{1}{2^n}\prod_{k=1}^n\frac{(6k-2)!(2k-1)!}{(4k-1)!(4k-2)!}.
\end{align}

The original proofs can be found in~\cite{Zeil-ASM,Kup-ASM} for $A_n$, and~\cite{MR1954236} for $A_{n}^V$.

\subsection{$O(n)$ Loop model}
\label{sub:O1}

In this subsection we briefly explain the $O(n)$ Loop model with periodic boundary conditions; for more details see~\cite{artic47, hdr, dG-review}. 
Let $n$ be an integer, and define a {\em state} as a column vector indexed by matchings of size $n$.

Let $e_i$ be the operator on matchings which creates a new arch at $(i,i+1)$, and join the vertices formerly linked to $i$ and $i+1$, as shown in the following examples:
\begin{align*}
 e_3 
\begin{tikzpicture}[scale=0.25, baseline=-3pt]
 \draw[arche] (0,0) .. controls (0,.5) and (1,.5) .. (1,0);
 \draw[arche] (2,0) .. controls (2,1.5) and (5,1.5) .. (5,0); 
 \draw[arche] (3,0) .. controls (3,.5) and (4,.5) .. (4,0);
 \draw[line] (-.5,0) -- (5.5,0);
\end{tikzpicture} = 
\begin{tikzpicture}[scale=0.25, baseline=-3pt]
 \draw[arche] (0,0) .. controls (0,.5) and (1,.5) .. (1,0);
 \draw[arche] (2,0) .. controls (2,1.5) and (5,1.5) .. (5,0); 
 \draw[arche] (3,0) .. controls (3,.5) and (4,.5) .. (4,0);
 \draw[line] (-.5,0) -- (5.5,0);
 \draw[arche] (0,0) -- (0,-1);
 \draw[arche] (1,0) -- (1,-1);
 \draw[arche] (2,0) .. controls (2,-.5) and (3,-.5) .. (3,0); 
 \draw[arche] (2,-1) .. controls (2,-.5) and (3,-.5) .. (3,-1); 
 \draw[arche] (4,0) -- (4,-1);
 \draw[arche] (5,0) -- (5,-1);
 \draw[line] (-.5,-1) -- (5.5,-1);
\end{tikzpicture} &= 
\begin{tikzpicture}[scale=0.25, baseline=-3pt]
 \draw[arche] (0,0) .. controls (0,.5) and (1,.5) .. (1,0);
 \draw[arche] (2,0) .. controls (2,.5) and (3,.5) .. (3,0); 
 \draw[arche] (4,0) .. controls (4,.5) and (5,.5) .. (5,0);
 \draw[line] (-.5,0) -- (5.5,0);
\end{tikzpicture}\\
 e_4
\begin{tikzpicture}[scale=0.25, baseline=-3pt]
 \draw[arche] (0,0) .. controls (0,.5) and (1,.5) .. (1,0);
 \draw[arche] (2,0) .. controls (2,1.5) and (5,1.5) .. (5,0); 
 \draw[arche] (3,0) .. controls (3,.5) and (4,.5) .. (4,0);
 \draw[line] (-.5,0) -- (5.5,0);
\end{tikzpicture} = 
\begin{tikzpicture}[scale=0.25, baseline=-3pt]
 \draw[arche] (0,0) .. controls (0,.5) and (1,.5) .. (1,0);
 \draw[arche] (2,0) .. controls (2,1.5) and (5,1.5) .. (5,0); 
 \draw[arche] (3,0) .. controls (3,.5) and (4,.5) .. (4,0);
 \draw[line] (-.5,0) -- (5.5,0);
 \draw[arche] (0,0) -- (0,-1);
 \draw[arche] (1,0) -- (1,-1);
 \draw[arche] (3,0) .. controls (3,-.5) and (4,-.5) .. (4,0); 
 \draw[arche] (3,-1) .. controls (3,-.5) and (4,-.5) .. (4,-1); 
 \draw[arche] (2,0) -- (2,-1);
 \draw[arche] (5,0) -- (5,-1);
 \draw[line] (-.5,-1) -- (5.5,-1);
\end{tikzpicture} &=
\begin{tikzpicture}[scale=0.25, baseline=-3pt]
 \draw[arche] (0,0) .. controls (0,.5) and (1,.5) .. (1,0);
 \draw[arche] (2,0) .. controls (2,1.5) and (5,1.5) .. (5,0); 
 \draw[arche] (3,0) .. controls (3,.5) and (4,.5) .. (4,0);
 \draw[line] (-.5,0) -- (5.5,0);
\end{tikzpicture}
\end{align*}
The operator $e_0$ creates an arch linking the positions $1$ and $2n$. 
Attached to these operators is the {\em Hamiltonian}
\[
 \mathcal{H}_{2n}=\sum_{i=0}^{2n-1} (1-e_i),
\]
where $1$ is the identity.  
$\mathcal{H}_{2n}$ acts naturally on states, and the groundstate $(\psi_\pi)_{\pi:|\pi|=n}$ attached to $\mathcal{H}_{2n}$ is defined as follows:

\begin{defi}[$\psi_\pi$]
\label{defi:psipi}
Let $n$ be a positive integer. 
We define the groundstate in the $O(n)$ Loop model as the  vector $\psi=(\psi_\pi)_{\pi:|\pi|=n}$ which is the solution of $\mathcal{H}_{2n}\psi=0$, normalized by $\psi_{()_n}=1$.
\end{defi}

By the Perron-Frobenius theorem, this is well defined. 
We have then the followings properties:

\begin{thm}
\label{th:propPsipi}
 Let $n$ be a positive integer.
\begin{itemize}
\item For any $\pi$,  $\psi_{r(\pi)}=\psi_{\pi^*}=\psi_{\pi}$.
\item The numbers $\psi_\pi$ are positive integers.
\item $\sum_\pi \psi_\pi = A_n$, where the sum is over matchings such that $|\pi|=n$.
\end{itemize}
\end{thm}

The stability by rotation and conjugation is clear from the symmetry of the problem. 
The integral property was proved in~\cite[Section 4.4]{artic43}, while the sum rule was proved in~\cite{artic31}.
The computation of this groundstate has received a lot of interest, mainly because of the Razumov--Stroganov (ex-)conjecture.

\subsection{The Razumov--Stroganov conjecture}

A simple computation shows that
\begin{align*}
 \psi_{
 \begin{tikzpicture}[scale=0.15]
   \draw[arche] (0,0) .. controls (0,.5) and (1,.5) .. (1,0);
   \draw[arche] (2,0) .. controls (2,.5) and (3,.5) .. (3,0); 
   \draw[arche] (4,0) .. controls (4,.5) and (5,.5) .. (5,0);
 \end{tikzpicture}}
&=2&
 \psi_{
 \begin{tikzpicture}[scale=0.15]
   \draw[arche] (0,0) .. controls (0,1.5) and (5,1.5) .. (5,0);
   \draw[arche] (1,0) .. controls (1,.5) and (2,.5) .. (2,0); 
   \draw[arche] (3,0) .. controls (3,.5) and (4,.5) .. (4,0);
 \end{tikzpicture}}
&=2 &
 \psi_{
 \begin{tikzpicture}[scale=0.15]
   \draw[arche] (0,0) .. controls (0,1) and (3,1) .. (3,0);
   \draw[arche] (1,0) .. controls (1,.5) and (2,.5) .. (2,0); 
   \draw[arche] (4,0) .. controls (4,.5) and (5,.5) .. (5,0);
 \end{tikzpicture}}
&=1\\ 
\psi_{
 \begin{tikzpicture}[scale=0.15]
   \draw[arche] (0,0) .. controls (0,.5) and (1,.5) .. (1,0);
   \draw[arche] (2,0) .. controls (2,1) and (5,1) .. (5,0); 
   \draw[arche] (3,0) .. controls (3,.5) and (4,.5) .. (4,0);
 \end{tikzpicture}}
&=1 &
 \psi_{
 \begin{tikzpicture}[scale=0.15]
   \draw[arche] (0,0) .. controls (0,1.5) and (5,1.5) .. (5,0);
   \draw[arche] (1,0) .. controls (1,1) and (4,1) .. (4,0); 
   \draw[arche] (2,0) .. controls (2,.5) and (3,.5) .. (3,0);
 \end{tikzpicture}}
&=1
\end{align*}
which are exactly the numbers that appear in the FPL counting:
\medskip

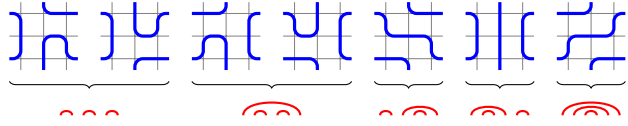
\[
  \begin{tikzpicture}[scale=.3]
  \draw[step=1,ajuda] (-.5,-.5) grid (2.5,2.5);
  \draw[FPL] (-.5,0) -- (0,0) -- (0,1) -- (0,2) -- (-.5,2);
  \draw[FPL] (1,-.5) -- (1,0) -- (1,1) -- (2,1) -- (2,0) -- (2.5,0);
  \draw[FPL] (1,2.5) -- (1,2) -- (2,2) -- (2.5,2);

  \draw[shift={(4,0)},step=1,ajuda] (-.5,-.5) grid (2.5,2.5);
  \draw[shift={(4,0)},FPL] (-.5,0) -- (0,0) -- (0,1) -- (0,2) -- (-.5,2);
  \draw[shift={(4,0)},FPL] (1,-.5) -- (1,0) -- (2,0) -- (2.5,0);
  \draw[shift={(4,0)},FPL] (1,2.5) -- (1,2) -- (1,1) -- (2,1) -- (2,2) -- (2.5,2);

  \draw[snake=brace,mirror snake] (-.5,-1)--(6.5,-1);

  \draw[shift={(1.75,-2.5)},arche] (0,0)..controls(0,.25)and(.5,.25)..(.5,0);
  \draw[shift={(1.75,-2.5)},arche] (1,0)..controls(1,.25)and(1.5,.25)..(1.5,0);
  \draw[shift={(1.75,-2.5)},arche] (2,0)..controls(2,.25)and(2.5,.25)..(2.5,0);

  \draw[shift={(8,0)},step=1,ajuda] (-.5,-.5) grid (2.5,2.5);
  \draw[shift={(8,0)},FPL] (-.5,0) -- (0,0) -- (0,1) -- (1,1) -- (1,0)-- (1,-.5);
  \draw[shift={(8,0)},FPL] (2.5,2) -- (2,2) -- (2,1) -- (2,0) -- (2.5,0);
  \draw[shift={(8,0)},FPL] (1,2.5) -- (1,2) -- (0,2) -- (-.5,2);

  \draw[shift={(12,0)},step=1,ajuda] (-.5,-.5) grid (2.5,2.5);
  \draw[shift={(12,0)},FPL] (-.5,0) -- (0,0) -- (1,0) -- (1,-.5);
  \draw[shift={(12,0)},FPL] (2.5,2) -- (2,2) -- (2,1) -- (2,0) -- (2.5,0);
  \draw[shift={(12,0)},FPL] (1,2.5) -- (1,2) -- (1,1) -- (0,1) -- (0,2) -- (-.5,2);

  \draw[snake=brace,mirror snake] (7.5,-1)--(14.5,-1);

  \draw[shift={(9.75,-2.5)},arche] (0,0)..controls(0,.75)and(2.5,.75)..(2.5,0);
  \draw[shift={(9.75,-2.5)},arche] (.5,0)..controls(.5,.25)and(1,.25)..(1,0);
  \draw[shift={(9.75,-2.5)},arche] (1.5,0)..controls(1.5,.25)and(2,.25)..(2,0);

  \draw[shift={(16,0)},step=1,ajuda] (-.5,-.5) grid (2.5,2.5);
  \draw[shift={(16,0)},FPL] (-.5,0) -- (0,0) -- (1,0) -- (1,-.5);
  \draw[shift={(16,0)},FPL] (-.5,2) -- (0,2) -- (0,1) -- (1,1) -- (2,1) -- (2,0) -- (2.5,0);
  \draw[shift={(16,0)},FPL] (1,2.5) -- (1,2) -- (2,2) -- (2.5,2);

  \draw[snake=brace,mirror snake] (15.5,-1)--(18.5,-1);

  \draw[shift={(15.75,-2.5)},arche] (0,0)..controls(0,.25)and(.5,.25)..(.5,0);
  \draw[shift={(15.75,-2.5)},arche] (1,0)..controls(1,.5)and(2.5,.5)..(2.5,0);
  \draw[shift={(15.75,-2.5)},arche] (1.5,0)..controls(1.5,.25)and(2,.25)..(2,0);

  \draw[shift={(20,0)},step=1,ajuda] (-.5,-.5) grid (2.5,2.5);
  \draw[shift={(20,0)},FPL] (-.5,0) -- (0,0) -- (0,1) -- (0,2) -- (-.5,2);
  \draw[shift={(20,0)},FPL] (1,-.5) -- (1,0) -- (1,1) -- (1,2) -- (1,2.5);
  \draw[shift={(20,0)},FPL] (2.5,0) -- (2,0) -- (2,1) -- (2,2) -- (2.5,2);

  \draw[snake=brace,mirror snake] (19.5,-1)--(22.5,-1);

  \draw[shift={(19.75,-2.5)},arche] (0,0)..controls(0,.5)and(1.5,.5)..(1.5,0);
  \draw[shift={(19.75,-2.5)},arche] (.5,0)..controls(.5,.25)and(1,.25)..(1,0);
  \draw[shift={(19.75,-2.5)},arche] (2,0)..controls(2,.25)and(2.5,.25)..(2.5,0);

  \draw[shift={(24,0)},step=1,ajuda] (-.5,-.5) grid (2.5,2.5);
  \draw[shift={(24,0)},FPL] (-.5,0) -- (0,0) -- (0,1) -- (1,1) -- (2,1) -- (2,2) -- (2.5,2);
  \draw[shift={(24,0)},FPL] (1,-.5) -- (1,0) -- (2,0) -- (2.5,0);
  \draw[shift={(24,0)},FPL] (1,2.5) -- (1,2) -- (0,2) -- (-.5,2);

  \draw[snake=brace,mirror snake] (23.5,-1)--(26.5,-1);

  \draw[shift={(23.75,-2.5)},arche] (0,0)..controls(0,.75)and(2.5,.75)..(2.5,0);
  \draw[shift={(23.75,-2.5)},arche] (.5,0)..controls(.5,.5)and(2,.5)..(2,0);
  \draw[shift={(23.75,-2.5)},arche] (1,0)..controls(1,.25)and(1.5,.25)..(1.5,0);
  \end{tikzpicture}
\]

\medskip
Razumov and Stroganov~\cite{RS-conj} noticed in 2001 that this seems to hold in general, and this was recently proved by Cantini and Sportiello~\cite{ProofRS}:

\begin{thm}[Stroganov--Razumov--Cantini--Sportiello Theorem]
\label{conj:rs}
 The groundstate components of the $O(n)$ Loop model count the number of FPL configurations: for any matching $\pi$,
\[
 \psi_\pi=A_{\pi}.
\]
\end{thm}

The proof of Cantini and Sportiello consists in verifying that the relations of Definition~\ref{defi:psipi} hold for the numbers $A_\pi$. 
We note also that the results of Theorem~\ref{th:propPsipi} are now a corollary of the Razumov--Stroganov conjecture.

\subsection{Matchings with nested arches and polynomials}

In~\cite{Zuber-conj}, Zuber computed some $\psi_{(\pi)_p}$ for some small matchings $\pi$, and $p=0,1,2,...$. Among other things, he conjectured the following:
\begin{thm}[{\cite{CKLN,artic47}}]
\label{zuber}
For any matching $\pi$ and $p$ a nonnegative integer, the quantity $A_{(\pi)_p}$ can be written in the following form:
\[
 A_{(\pi)_p}=\frac{P_\pi (p)}{d(\pi)!},
\]
where $P_\pi (p)$ is a polynomial in $p$ of degree $d(\pi)$ with integer coefficients, and leading coefficient equal to $d(\pi)!/H_{Y(\pi)}$.
\end{thm}

This was proved first by Caselli, Krattenthaler, Lass and Nadeau in~\cite{CKLN} for $A_{(\pi)_p}$, and by Fonseca and Zinn-Justin in~\cite{artic47} for $\psi_{(\pi)_p}$. 
Because of this polynomiality property, we introduce the following notations:

\begin{defi}[$A_\pi(t)$ and $\psi_\pi(t)$] 
We let $A_\pi(t)$ (respectively  $\psi_\pi(t)$) be the polynomial in $t$ such that  $A_\pi(p)=A_{(\pi)_p}$ (resp. $\psi_\pi(p)=\psi_{(\pi)_p}$) for all integers $p\geq 0$.
\end{defi}

By the Razumov--Stroganov conjecture~\ref{conj:rs} one has clearly for all $\pi$:
\[
 A_\pi(t)=\psi_\pi(t).
\]
We introduced two different notations so that the origin of the quantities involved becomes clearer; in most of this paper however we will only use the notation $\psi_\pi(t)$. 
The following proposition sums up some properties of the polynomials.

\begin{prop}
\label{prop:polynomials}
The polynomial $\psi_\pi(t)$ has degree $d(\pi)$ and leading coefficient $1/H_{Y(\pi)}$. 
Furthermore, we have $\psi_\pi(t)=\psi_{\pi^*}(t)$, and $\psi_{(\pi)_\ell}(t)=\psi_{\pi}(t+\ell)$ for any nonnegative integer $\ell$.
\end{prop}

  The first part comes from Theorem~\ref{zuber}, while the rest is clear when $t$ is a nonnegative integer and thus holds true in general by polynomiality in $t$. 

\section{Conjectures}\label{sec:conj}

The aim of this article is to prove two conjectures presented in~\cite{negative} about the polynomials $\psi_\pi (t)$ for negative $t$.
In fact, when computing these quantities, it is natural to add an extra parameter $\tau$, \ie there is a bivariate polynomial $\psi_\pi (\tau,t)$ which has the same properties as $\psi_\pi (t)$ and in the limit $\tau=1$, it coincides with $\psi_\pi (t)$.

In Section~\ref{sec:CPL_multi}, where we explain how to compute the $\psi_\pi (t)$, the origin of this parameter will be made more clear.
For now, it will be enough to think at this parameter as a refinement.

\subsection{Integer roots}

Let $\pi$ be a matching, represented by a link pattern, and $|\pi|=n$ its number of arches.
Define $\hat{x}:=2n+1-x$.

\begin{defi}[$m_p (\pi)$]
Let $p$ be an integer between $1$ and $n-1$. 
We consider the set $\mathcal{A}_p^L (\pi)$ of arches $\{a_1,a_2\}$ such that $a_1\leq p$ and $p<a_2<\hat{p}$, and the set $\mathcal{A}_p^R (\pi)$ of arches $\{a_1,a_2\}$ such that $p<a_1<\hat{p}$ and $a_2\geq \hat{p}$. 
It is clear that $\left|\mathcal{A}_p^L(\pi)\right|+\left|\mathcal{A}_p^R(\pi)\right|$ is a even nonnegative integer, and we can thus define the nonnegative integer
\[
 m_p(\pi) := \frac{\left|\mathcal{A}_p^L(\pi)\right|+\left|\mathcal{A}_p^R(\pi)\right|}{2}.
\]
\end{defi}

For example, let $\pi$ be the following matching with eight arches. 
For $p=4$, we get $\left|\mathcal{A}_p^L(\pi)\right|=3$ and $\left|\mathcal{A}_p^R(\pi)\right|=1$, which count arches between the regions (O) and (I), thus $m_4(\pi)=2$.
In the figure on the right we give an alternative representative by folding the link pattern, it is then clear that $m_p(\pi)$ is half of the number of arches linking (O) with (I).

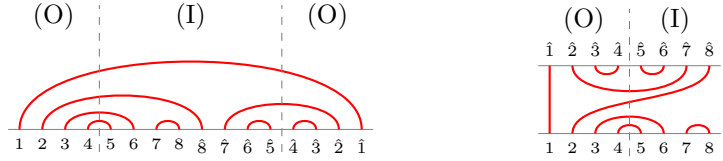
\[
\begin{tikzpicture}[scale=.3]
\draw[ajuda] (-.5,0)--(15.5,0);
\draw[arche] (0,0) .. controls (0,4) and (15,4) .. (15,0);
\draw[arche] (1,0) .. controls (1,2) and (8,2) .. (8,0);
\draw[arche] (2,0) .. controls (2,1) and (5,1) .. (5,0);
\draw[arche] (3,0) .. controls (3,.5) and (4,.5) .. (4,0);
\draw[arche] (6,0) .. controls (6,.5) and (7,.5) .. (7,0);
\draw[arche] (9,0) .. controls (9,1.5) and (14,1.5) .. (14,0);
\draw[arche] (10,0) .. controls (10,.5) and (11,.5) .. (11,0);
\draw[arche] (12,0) .. controls (12,.5) and (13,.5) .. (13,0);
\node[below] at (0,0) {\tiny{$1$}};
\node[below] at (1,0) {\tiny{$2$}};
\node[below] at (2,0) {\tiny{$3$}};
\node[below] at (3,0) {\tiny{$4$}};
\node[below] at (4,0) {\tiny{$5$}};
\node[below] at (5,0) {\tiny{$6$}};
\node[below] at (6,0) {\tiny{$7$}};
\node[below] at (7,0) {\tiny{$8$}};
\node[below] at (8,0) {\tiny{$\hat{8}$}};
\node[below] at (9,0) {\tiny{$\hat{7}$}};
\node[below] at (10,0) {\tiny{$\hat{6}$}};
\node[below] at (11,0) {\tiny{$\hat{5}$}};
\node[below] at (12,0) {\tiny{$\hat{4}$}};
\node[below] at (13,0) {\tiny{$\hat{3}$}};
\node[below] at (14,0) {\tiny{$\hat{2}$}};
\node[below] at (15,0) {\tiny{$\hat{1}$}};

\draw[gray,dashed] (3.5,-1) -- (3.5,5.5);
\draw[gray,dashed] (11.5,-1) -- (11.5,5.5);

\node at (1.5,5) {(O)};
\node at (7.5,5) {(I)};
\node at (13.5,5) {(O)};
\end{tikzpicture}
\qquad\qquad\qquad
\begin{tikzpicture}[scale=.3]
\draw[ajuda] (-.5,0)--(7.5,0);
\draw[ajuda] (-.5,3)--(7.5,3);

\draw[arche] (0,0)--(0,3);
\draw[arche] (1,0) .. controls (1,1.75) and (7,1.25) .. (7,3);
\draw[arche] (2,0) .. controls (2,1) and (5,1) .. (5,0);
\draw[arche] (3,0) .. controls (3,.5) and (4,.5) .. (4,0);
\draw[arche] (6,0) .. controls (6,.5) and (7,.5) .. (7,0);
\draw[arche] (6,3) .. controls (6,1.5) and (1,1.5) .. (1,3);
\draw[arche] (5,3) .. controls (5,2.5) and (4,2.5) .. (4,3);
\draw[arche] (3,3) .. controls (3,2.5) and (2,2.5) .. (2,3);

\node[below] at (0,0) {\tiny{$1$}};
\node[below] at (1,0) {\tiny{$2$}};
\node[below] at (2,0) {\tiny{$3$}};
\node[below] at (3,0) {\tiny{$4$}};
\node[below] at (4,0) {\tiny{$5$}};
\node[below] at (5,0) {\tiny{$6$}};
\node[below] at (6,0) {\tiny{$7$}};
\node[below] at (7,0) {\tiny{$8$}};
\node[above] at (0,3) {\tiny{$\hat{1}$}};
\node[above] at (1,3) {\tiny{$\hat{2}$}};
\node[above] at (2,3) {\tiny{$\hat{3}$}};
\node[above] at (3,3) {\tiny{$\hat{4}$}};
\node[above] at (4,3) {\tiny{$\hat{5}$}};
\node[above] at (5,3) {\tiny{$\hat{6}$}};
\node[above] at (6,3) {\tiny{$\hat{7}$}};
\node[above] at (7,3) {\tiny{$\hat{8}$}};

\draw[gray,dashed] (3.5,-1) -- (3.5,5.5);

\node at (1.5,5) {(O)};
\node at (5.5,5) {(I)};
\end{tikzpicture}
\]
The reader can check that $m_p=0,1,2,2,2,1,1$ for $p=1,\ldots,7$.

It was conjectured in~\cite{negative} that these numbers correspond to the multiplicity of the real roots of $\psi_\pi (t)$:

\begin{conj}\label{conj:realroots}
 All the real roots of the polynomials $\psi_{\pi}(t)$ are negative integers, and $-p$ appears with multiplicity $m_p(\pi)$. Equivalently, we have a factorization:
\[
 \psi_{\pi}(t) = \frac{1}{|d(\pi)|!} \cdot \left(\prod_{p=1}^{|\pi|-1} (t+p)^{m_p(\pi)}\right)\cdot Q_{\pi} (t),
\]
where $Q_{\pi} (t)$ is a polynomial with integer coefficients and no real roots. 
\end{conj}

In the context of the $O(n)$ Loop model, it is normal to have an extra parameter $\tau$. We have then $\psi_{\pi} (\tau,t)$ which coincides with $\psi_{\pi} (\tau)$ whenever $\tau = 1$. 
The previous conjecture seems to hold, the only difference being that $Q_\pi$ depends now on $\tau$.

In Section~\ref{sec:deco}, we prove a weaker version of this conjecture: $\psi_\pi (\tau,-p)=0$ if $m_p (\pi)\neq 0$.

\subsection{Values at negative $p$}

We are now interested in the value of the polynomial $\psi_\pi (\tau,-p)$ for integer values $0 \leq p \leq n$.
It has been already conjectured that it vanishes if $m_p (\pi) \neq 0$.

So let $\pi$ be a matching and $p$ such that $m_p (\pi)$ vanishes. 
It means that there are no arches that link the outer part with the inner part of $\pi$.
\Ie we can define a matching sitting in the outer part (denote it by $\alpha$) and an other in the inner part (denote it by $\beta$), as shown in the picture:
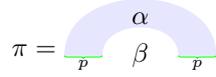
\[
 \pi=
 \begin{tikzpicture}[scale=.25,baseline=0pt]
  \fill [blue!10!white] (4,0) ..  controls (4,4) and (-4,4) .. (-4,0) -- (-2,0) .. controls (-2,2) and (2,2) .. (2,0) -- cycle;
  \draw [green, snake=brace, mirror snake, segment amplitude=1pt] (-4,0) -- (-2,0);
  \draw [green, snake=brace, mirror snake, segment amplitude=1pt] (2,0) -- (4,0);
  \draw [black] (-3,-.5) node {\tiny $p$};
  \draw [black] (3,-.5) node {\tiny $p$};
  \draw [black] (0,0) node {$\beta$};
  \draw [black] (0,2) node {$\alpha$};
 \end{tikzpicture}
\]
we introduce the notation $\pi = \alpha \circ \beta$ to describe this situation.
We need one more definition:
\begin{defi}[$g_\pi$] For any matching $\pi$ we define 
\[
g_\pi (\tau):=\psi_{\pi}(\tau,-|\pi|),
\]
and $g_\pi := g_\pi (1)$.
\end{defi}

We are now ready to present the main result of this article:
\begin{thm}[Generalization of Conjecture 3.8 of~\cite{negative}]\label{thm:dec}
 Let $\pi$ be a matching and $p$ be an integer between $1$ and $|\pi|-1$ such that $m_p(\pi)=0$, and write $\pi=\alpha \circ \beta$ with $|\alpha|=p$. We then have the following factorization:
 \[
  \psi_{\pi}(\tau,-p)=  g_\alpha(\tau) \psi_{\beta}(\tau).
 \]
\end{thm}

Notice that we are reducing the number of unknowns from $c_n$ to $c_p+c_{n-p}-1$.
The proof is postponed to Section~\ref{sec:deco}.

\subsection{Sum rules}
It has been proved in~\cite{tese} that these numbers $g_\pi (\tau)$ have some interesting properties.
For example $g_{\pi} (-\tau) = (-1)^{d(\pi)} g_{\pi} (\tau)$ and
\begin{thm}[\cite{tese}]\label{thm:sum_Gtau}
We have the sum rule:
\[
 \sum_{\pi:|\pi|=n} g_\pi (\tau) = \sum_{\pi:|\pi|=n} \psi_\pi (-\tau).
\]
\end{thm}

This can be used to partially prove Conjecture 3.11 of~\cite{negative}.
Notice that, according to the Conjecture~\ref{conj:realroots}, $(-1)^{d(\pi)}g_\pi = |g_\pi|$.
\begin{thm}\label{thm:sum_G}
For any positive integer $n$, we have
\begin{align}
\sum_{\pi:|\pi|=n} (-1)^{d(\pi)} g_\pi &= A_n \label{eq:sum_G_1} \\
\sum_{\pi:|\pi|=n} g_\pi &= (-1)^{\frac{n(n-1)}{2}}\left(A^V_n\right)^2 \label{eq:sum_G_2},
\end{align}
where $d(\pi)$ is the number of boxes of the Young diagram $Y(\pi)$.
\end{thm}

The first equation~\eqref{eq:sum_G_1} has been proved in~\cite{tese}, and it follows from Theorem~\ref{thm:sum_Gtau}. Here we will prove the second equation~\ref{eq:sum_G_2}. The point is that $\sum_{\pi} g_{\pi} = \sum_{\pi} \psi_{\pi} (-1)$ is equivalent to the minus enumeration of TSSCPP which appears in Di Francesco's article~\cite{DF-qKZ-TSSCPP}, and this can be computed, see Section~\ref{sec:-enum} for a better explanation.

\section{Multivariate solutions of the $O(n)$ Loop model}\label{sec:CPL_multi}

In this section, we briefly describe a multivariate version of the $O(n)$ Loop model. 
This version, although more complicated, is useful to the proof of Theorem~\ref{thm:dec}.

The $O(n)$ model is an integrable model, meaning that there is an operator called $\check{R}$-Matrix which obeys to the Yang--Baxter equation.
We can add new parameters $\{z_1,z_2,\ldots,z_{2n}\}$, called spectral parameters, which will characterize each column.
In this multivariate setting, the groundstate depends on the $2n$ spectral parameters (and in an extra parameter $q$), in fact the components of the groundstate can be normalized such that they are homogeneous polynomials $\Psi_\pi (z_1,\ldots,z_{2n})$ of degree $n(n-1)$.
The important fact about these solutions is that we recover the solutions of the $O(n)$ model, as stated in Section~\ref{sub:O1}, in the limit $z_i=1$ for all $i$, and $q=e^{2\pi i/3}$.

We shall not describe this model in detail here, such detailed description can be found in~\cite{artic47,hdr,artic43,tese}.

In order to simplify notation we will use $z=\{z_1,\ldots,z_{2n}\}$, thus we will often write $\Psi_\pi (z)$. 
Notice that these polynomials depend on $q$, but we omit this dependence.

\subsection{The quantum Knizhnik--Zamolodchikov equation}\label{sec:qKZ}

The groundstate of the multivariate $O(n)$ model is known to solve the quantum Knizhnik--Zamolodchikov (qKZ) equation in the special value $q=e^{2\pi i/3}$.
See a complete explanation in~\cite{hdr,tese}.

The qKZ equation was firstly introduced in a paper by Frenkel and Reshetkhin~\cite{FR-qkz}.
Here we use the version introduced by Smirnov~\cite{Smi}. 
Let the $\check{R}$-Matrix be the following operator, 
\[
 \check{R}_i(z_i,z_{i+1}) = \frac{q z_{i+1}-q^{-1}z_i}{q z_i-q^{-1} z_{i+1}} Id + \frac{z_{i+1}-z_i}{q z_i - q^{-1}z_{i+1}} e_i.\label{eq:R-matrix}
\]

The quantum Knizhnik--Zamolodchikov equation:
\begin{itemize}
 \item The \emph{exchange} equation:
\begin{equation}\label{eq:exc}
\check{R}_i(z_i,z_{i+1}) \Psi(z_1,\ldots,z_i,z_{i+1},\ldots,z_{2n}) = \Psi(z_1,\ldots,z_{i+1},z_i.\ldots,z_{2n}),
\end{equation}
for $i=1,\ldots,2n$.

 \item The \emph{rotation} equation:
\begin{equation}\label{eq:rot}
 \rho^{-1} \Psi (z_1,z_2,\ldots,z_{2n})=\kappa \Psi (z_2,\ldots,z_{2n},s z_1),
\end{equation}
where $\kappa$ is a constant such that $\rho^{-2n}=1$. In our case $s=q^6$ and $\kappa=q^{3(n-1)}$.
\end{itemize}

\subsection{Solutions of the quantum Knizhnik--Zamolodchikov equation}\label{sec:sol_qKZ}

We start for pointing down some properties of the solutions of the qKZ equation without proof.
\begin{itemize}
\item The solutions are homogeneous polynomials in $2n$ variables;
\item The total degree is $n(n-1)$ and the individual degree in each $z_i$ is $n-1$;
\item They obey to the \emph{wheel condition}:
\[
 \left.P(z_1,\ldots,z_{2n})\right|_{z_k=q^2 z_j = q^4 z_i} =0 \qquad \forall \, k>j>i.
\] 
\end{itemize}

In fact these three properties define a vector space:

\begin{defi}[$\mathcal{V}_n$]
 We define $\mathcal{V}_n$ as the vector space of all homogeneous polynomials in $2n$ variables, with total degree $\delta=n(n-1)$ and individual degree $\delta_i=n-1$ in each variable which obey to the \emph{wheel condition}.
\end{defi}

This vector space has dimension $c_n$, exactly the number of matchings of size $|\pi|=n$. 
Moreover, the polynomials $\Psi_\pi(z)$ verify the following important lemma:

\begin{lemma}[\cite{artic41}]\label{lem:dual}
Let $q^\epsilon=\{q^{\epsilon_1},\ldots,q^{\epsilon_{2n}}\}$, where $\epsilon_i=\pm 1$ are such that changing $q^{-1}$ into ``$($'' and changing $q$ into ``$)$'' gives a valid parenthesis word $\pi(\epsilon)$. Then 
\[
 \Psi_\pi(q^\epsilon) = \tau^{d(\pi)} \delta_{\pi,\epsilon},
\]
where $\delta_{\pi,\epsilon}=1$ when we have $\pi(\epsilon)=\pi$. $\tau$ is related with $q$ by the formula $\tau=-q-q^{-1}$.
\end{lemma}

Since there are $c_n$ polynomials $\Psi_\pi(z)$, this lemma shows that these polynomials form a basis of $\mathcal{V}_n$. 
Thus a polynomial in this space is determined by its value on these points $q^\epsilon$.

\subsection{A different approach}\label{sec:base_a}
We now define another set of polynomials, introduced in~\cite{artic41}, $\Phi_a (z_1,\ldots,z_{2n})$ (indexed by the increasing sequences defined earlier), by the integral formula:
\begin{multline}\label{eq:qKZ_var}
 \Phi_a(z)= k_n
 \prod_{1\le i<j\le 2n} (qz_i-q^{-1}z_j)\\ \times \oint\ldots\oint \prod_{i=1}^n \frac{dw_i}{2\pi i} \frac{\prod_{1\le i<j\le n}(w_j-w_i)(qw_i-q^{-1}w_j)}{\prod_{1\le k\leq a_i}(w_i-z_k)\prod_{a_i<k\le 2n}(qw_i-q^{-1}z_k)},
\end{multline}
where the integral is performed around the $z_i$ but not around $q^{-2} z_i$, and $k_n=(q-q^{-1})^{-n(n-1)}$.

It is relatively easy to check that these polynomials belong to the vector space $\mathcal{V}_n$, so we can write:
\[
\Phi_a (z)=\sum_\pi C_{a,\pi} (\tau) \Psi_\pi (z),
\]
where $C_{a,\pi}(\tau)$ are the coefficients given by the formula:
\[
 \Phi_a (q^\epsilon) = \tau^{d(\epsilon)} C_{a,\epsilon} (\tau).
\]

An algorithm to compute the coefficients $C_{a,\pi} (\tau)$ is given in~\cite[Appendix A]{artic41}. 
We just need the following facts:
\begin{prop}[{\cite[Lemma 3]{artic47}}] 
\label{prop:Bases}
 Let $a$ and $\pi$ be two matchings. Then we have:
\[
 C_{a,\pi}(\tau)=\begin{cases}
             0 & \textrm{if } \pi \npreceq a;\\
             1 & \textrm{if } \pi=a;\\
             P_{a,\pi} (\tau) & \textrm{if } \pi \prec a,
            \end{cases}
\]
where $P_{a,\pi}(\tau)$ is a polynomial in $\tau$ with degree $\leq d(a)-d(\pi)-2$.
\end{prop}

Moreover, we have
\begin{equation}
\label{eq:capi-tau}
 C_{a,\pi}(\tau)=(-1)^{d(a)-d(\pi)} C_{a,\pi}(-\tau),
\end{equation}
since it is a product of polynomials $U_s$ in $\tau$ with degree of the form $d(a)-d(\pi)-2k$, $k\in \mathbb{N}$, and parity given by $d(a)-d(\pi)$: this is an easy consequence of~\cite[p.12 and Appendix C]{artic47}. 

By abuse of notation, we write $(a)_p$ to represent $\{1,\ldots,p,p+a_1,\ldots,p+a_n\}$, since this corresponds indeed to adding $p$ nested arches to $\pi(a)$ via the bijections of Section~\ref{sec:defi}. Then 
one easy but important lemma for us is the following:

\begin{lemma}[{\cite[Lemma 4]{artic47}}]
\label{lem:sta} 
The coefficients $C_{a,\pi}(\tau)$ are stable, that is:
  \[
   C_{(a)_p,(\pi)_p}(\tau)=C_{a,\pi}(\tau) \qquad \forall p\in \mathbb{N}.
 \]
\end{lemma}

We remark that Proposition~\ref{prop:Bases}, Equation~\eqref{eq:capi-tau} and Lemma~\ref{lem:sta} also hold for the coefficients $\C_{a,\pi}(\tau)$ of the inverse matrix. 

\subsection{The homogeneous limit}\label{sec:pol_qKZ}

The bivariate polynomials $\psi_{\pi}(\tau,t)$ are defined as the homogeneous limit of the previous multivariate polynomials (\ie $z_i=1$ for all $i$). 
Though we are mostly interested in the case $\tau=1$, since we recover the groundstate $\psi_{\pi}(t)=\psi_{\pi}(1,t)$, as explained in~\cite{hdr}, the variable $\tau$ it will be useful for the proofs presented here.

Define $\phi_a (\tau):=\Phi_a (1,\ldots,1)$\footnote{Notice that $\Phi_a(z)$ depends on $q$, even if we do not write it explicitly.}. 
Using variable transformation 
\[
u_i = \frac{w_i-1}{q w_i - q^{-1}},
\]
we obtain the formula:
\begin{equation}
 \phi_{a}(\tau) = \oint \ldots \oint \prod_i \frac{du_i}{2 \pi i u_i^{a_i}} \prod_{j>i} (u_j-u_i) (1+\tau u_j+u_i u_j). 
\end{equation}

Thus, we can then obtain the $\psi_\pi(\tau)$ via the matrix $C(\tau)$:
\begin{align} 
 \phi_a (\tau)=&\sum_\pi C_{a,\pi}(\tau) \psi_\pi (\tau);\label{eq:psiphi1}\\
 \psi_\pi (\tau)=&\sum_a \C_{\pi,a}(\tau) \phi_a (\tau).\label{eq:psiphi2}
\end{align}

Let $\hat{a}_i$ be the components of $(a)_p$. 
Now
\begin{align*}
 \phi_{(a)_p} (\tau) =& \oint\ldots\oint \prod_i^{n+p} \frac{du_i}{2\pi i u_i^{\hat{a}_i}} \prod_{j>i} (u_j-u_i)(1+\tau u_j+u_i u_j)\\
 =& \oint\ldots\oint \prod_i^n \frac{du_i}{2\pi i u_i^{a_i}} (1+\tau u_i)^p \prod_{j>i} (u_j-u_i)(1+\tau u_j+u_i u_j),
\end{align*}
where we integrated in the first $p$ variables and renamed the rest $u_{p+i}\mapsto u_i$. 
This is a polynomial in $p$, and we will naturally note $\phi_{a} (\tau,t)$ the polynomial such that $\phi_{a} (\tau,p)=\phi_{(a)_p}(\tau)$.

Finally, from Equation~\eqref{eq:psiphi2} and Lemma~\ref{lem:sta} we obtain the fundamental equation
\begin{equation}
\label{eq:psitauphitau}
 \psi_\pi (\tau,t) = \sum_a \C_{\pi,a}(\tau) \phi_a (\tau,t).
\end{equation}

In the special case $\tau=1$, we write $C_{a,\pi}=C_{a,\pi}(1)$, $\phi_a (t)=\phi_a (1,t)$ and thus
\[
 A_\pi(t)=\psi_\pi (t) = \sum_a \C_{\pi,a} \phi_a (t),
\]
thanks to the Razumov--Stroganov conjecture~\ref{conj:rs}.

\section{Decomposition formula}\label{sec:deco}

The aim of this section is to prove Theorem~\ref{thm:dec}. 
With this in mind, we introduce two new multivariate polynomials $\Psi_{\pi,-p} (z)$ and $G_\pi (z)$, which generalize the quantities $\psi_\pi (\tau,-p)$ and $g_\pi (\tau)$ respectively.

\subsection{New polynomials}
Let $\pi$ be a matching with size $|\pi|=n$ and $p$ be a nonnegative integer less or equal than $n$. 
We require that the new object $\Psi_{\pi,-p} (z)$ has the following essential properties:
\begin{itemize}
\item It generalizes $\psi_\pi (\tau,-p)$, \ie $\psi_\pi (\tau,-p) = \Psi_{\pi,-p} (1,\ldots,1)$;
\item When $p=0$, we have $\Psi_{\pi,0} (z) = \Psi_\pi (z)$, justifying the use of the same letter;
\item They are polynomials on $z_i$.
\end{itemize}
Thus, we can define a multivariate version of $g_\pi (\tau)$, by 
\begin{equation}
G_\pi (z_1,\ldots,z_{2n}) := \Psi_{\pi,-|\pi|} (z_1,\ldots,z_{2n}),
\end{equation}
such that $g_\pi (\tau) = G_\pi (1,\ldots,1)$.

Surprisingly enough, using these new polynomials we can state a theorem equivalent to Theorem~\ref{thm:dec}:
\begin{thm}\label{thm:dec_mult}
 Let $\pi$ be a matching and $p$ be an integer between $1$ and $|\pi|-1$ such that $m_p(\pi)=0$, and write $\pi=\alpha \circ \beta$ with $|\alpha|=p$. We then have the following factorization:
 \[
  \Psi_{\pi,-p}(z_1,\ldots,z_{2n})=  G_\alpha(z_1,\ldots,z_{p},z_{\hat{p}},\ldots,z_{\hat{1}}) \Psi_{\beta}(z_{p+1},\ldots,z_{\hat{p}-1}).
 \]
\end{thm}

In what follows we use the short notation $z^{(O)}$ for the outer variables $\{z_1,\ldots,z_p,\allowbreak z_{\hat{p}},\ldots,z_{\hat{1}}\}$ and $z^{(I)}$ for the inner variables $\{z_{p+1},\ldots,z_{\hat{p}-1}\}$.

\subsection{A contour integral formula}

We will follow the same path as in Section~\ref{sec:base_a}.
That is, we introduce a new quantity $\Phi_{a,-p} (z)$ defined by a multiple contour integral formula, and after we can obtain $\Psi_{\pi,-p} (z)$ by:
\begin{equation}
\Psi_{\pi,-p} (z) := \sum_a \C_{\pi,a}(\tau) \Phi_{a,-p} (z).
\end{equation}

This new quantity $\Phi_{a,-p} (z)$ must be a generalization of the formula $\Phi_a(z)$, let $\hat{\jmath}=2n-j+1$:
\begin{align*}
 \Phi_{a,-p}(z)&:=k_n \prod_{1\leq i,j \leq p} (q z_i -q^{-1} z_j) \prod_{2 \leq i,j \leq p} (q z_i - q^{-1}z_{\hat{\jmath}}) \\
  &\qquad \prod_{i=1}^p \prod_{j=p+1}^{2n-p} (q z_i -q^{-1} z_j) \prod_{p<i<j< \hat{p}} (q z_i -q^{-1} z_j) \\
&\qquad \oint \ldots \oint \prod_{i=1}^n \frac{dw_i}{2\pi i} \frac{\prod_{j>i} (w_j-w_i)(qw_i-q^{-1}w_j)}{\prod_{j\leq a_i} (w_i-z_j) \prod_{j>a_i} (qw_i-q^{-1}z_j)} \prod_{j=1}^p \frac{qw_i-q^{-1}z_{\hat{\jmath}}}{qz_j - q^{-1}w_i},
\end{align*}
where $k_n$ is a normalization constant:
\[
 k_n = 
  \begin{cases}
   (q-q^{-1})^{-n(n-1)}&\text{if }p=0;\\
   (q-q^{-1})^{-(p-1)^2 -(n-p)(n-p-1)}& \text{otherwise},
  \end{cases} 
\] 
and the contours of integration surround all $z_i$ but not $q^{\pm 2} z_i$. 
This means, that integrating is equivalent to choose all possible combinations of poles $(w_k - z_i)^{-1}$ for all $k\leq n$ such that $i\leq a_k$. 
Notice that the presence of the Vandermonde implies that we cannot chose the same pole twice.

In the homogeneous limit $z_i=1$ for all $i$, we get:
\[
 \Phi_{a,-p} (1,\ldots,1)=\oint \ldots \oint \prod_i \frac{du_i}{2\pi i u_i^{a_i}} (1+\tau u_i)^{-p} \prod_{j>i} (u_j-u_i)(1+\tau u_j +u_i u_j), 
\]
which is precisely $\phi_a (\tau,-p)$.
In fact, this is the main reason for the formula presented here.

Therefore, we can write 
\begin{equation}
 G_\pi (z)=\sum_a \C_{\pi,a}(\tau) \Phi_{a,-|\pi|} (z),
\end{equation}
where the sum runs over all matchings $a$.
In fact, we will use this equation as definition of $G_\pi (z)$.

\subsection{Some properties of $\Phi_{a,-p} (z)$}

In Section~\ref{sec:base_a}, we have seen that $\Phi_a (z)$ are useful because they are homogeneous polynomials with a certain degree which obey to the \emph{wheel condition}, thus they span $\mathcal{V}_n$ and we can expand $\Psi_\pi (z)$ as a linear combination of $\Phi_a (z)$.
We hope that we can find some properties of $\Phi_{a,-p} (z)$ which allow us to apply similar methods.
Let us start with the polynomiality:

\begin{prop}[Polynomiality]\label{prop:poly}
 The function $\Phi_{a,-p} (z_1,\ldots,z_{2n})$ is a homogeneous polynomial on the variables $z_i$.
\end{prop}

\begin{proof}
It is obvious that $\Phi_{a,-p} (z)$ can be written as a ratio of two polynomials.
Thus, to prove that this is a polynomial, it is enough to prove that there are no poles.
The proof is straightforward but tedious, thus we shall not repeat on the body of this paper, see the Appendix~\ref{sec:poly} for the details.

The fact that it is homogeneous is obvious from the definition, once we know that it is a polynomial.
\end{proof}

\begin{prop}[The individual degree]
 The degree of the polynomials $\Psi_{a,-p} (z)$ at a single variable $z_i$ is given by: 
\[
 \delta_i = 
\begin{cases}
  0&\text{if }i=1\text{ or }i=\hat{1};\\
  p-1&\text{if }1<i\leq p \text{ or }\hat{p}\leq i <\hat{1};\\
  n-p-1&\text{if }p<i<\hat{p}.
\end{cases}
\]
\end{prop}

\begin{proof}
For a certain variable $z_i$ two things can happen when we perform the contour integration.
Either we chose a pole $(w_k-z_i)^{-1}$ for some $k=1,\ldots,n$ or not.
It is enough to compute the degree in both cases and we arrive to the desired result.
\end{proof}

\begin{prop}[The combined degree]
The degree of the polynomials $\Psi_{a,-p} (z)$ in the outer variables is $(p-1)^2$, in the inner variables is $(n-p)(n-p-1)$ and in all variables is $(p-1)^2 +(n-p)(n-p-1)$.
\end{prop}

\begin{proof}
The total degree in all variables is easy to compute: $\delta=(p-1)^2 +(n-p)(n-p-1)$.

The total degree in the inner variables (respectively outer variables) is more complex. 
Assume that we choose $\alpha$ inner poles $(w_i-z_j)^{-1}$ for $p<j<\hat{p}$ (respectively outer poles $(w_i-z_j)^{-1}$ for $j\leq p$ or $j\geq \hat{p}$).
The degree is $\alpha(2n-2p-\alpha)-n+p$ (respectively $\alpha (2p-\alpha)-2p+1$).

The maximum is when $\alpha=n-p$ (respectively $\alpha=p$), and it is equal to $(n-p)(n-p-1)$ (resp. $(p-1)^2$).
\end{proof}

Notice that if we choose the maximum in both sets (the outer and the inner variables) we obtain $(p-1)^2+(n-p)(n-p-1)$, which is equal to $\delta$. 
Thus we can write:
\begin{equation}
 \Phi_{a,-p}(z)=\sum_i P_i (z^{(O)}) Q_i (z^{(I)}),
\end{equation}
where $P_i$ and $Q_i$ are polynomials of total degree $(p-1)^2$ and $(n-p)(n-p-1)$ respectively. 

\begin{prop}[The \emph{wheel condition}]
Let $z_k=q^2 z_j=q^4 z_i$ for $p<i<j<k<\hat{p}$.
Thus, $\Phi_{a,-p}(z)=0$.
\end{prop} 

\begin{proof}
The term $\prod_{p<i<j<\hat{p}} (qz_i-q^{-1}z_j)$ contain two zeros (if $z_k=q^2 z_j = q^4 z_i$), if we want to prove the \emph{wheel condition} it is enough to prove that it is impossible to cancel both at the same time.
In order to cancel the zero $(qz_j-q^{-1}z_k)$, we need to chose a pole $(w_l-z_j)^{-1}$ for some $l$ such that $j\leq a_l <k$.
In the same way we must chose $(w_m - z_i)^{-1}$ for some $m$ such that $i\leq a_m <j$.
But this implies that $m<l$, so $(qw_m-q^{-1}w_l)$ will be zero, making the whole expression to vanish.
\end{proof}

Therefore the inner parts $Q_i (z_{p+1},\ldots,z_{\hat{p}-1})$ are homogeneous polynomials with total degree $(n-p)(n-p-1)$ and satisfy the \emph{wheel condition}, so they belong to the vector space $\mathcal{V}_{n-p}$, that is:
\begin{equation}
\Phi_{a,-p}(z)=\sum_{\beta:|\beta|=n-p} P_{a;\beta} (z^{(O)}) \Psi_\beta (z^{(I)}),
\end{equation}
where the sum runs over all matchings $\beta$ of size $n-p$.

\subsection{Computing $P_{a;\beta}$}
In order to compute these polynomials, we need a new definition: 

\begin{defi}
Let $a=\{a_1,\ldots,a_n\}$ be a matching of size $n$.
Separate it into two parts: the inner part composed by all $p<a_i<\hat{p}$ subtracted by $p$, and the outer part composed by all $a_i\leq p$, and the $a_i\geq\hat{p}$ subtracted by $2(n-p)$.
Let $c$ be the inner part and $b$ the outer part. 
Moreover if the number of elements of $c$ is bigger than $n-p$ by $s$ we add $s$ times $p$ to the outer part.
$b$ and $c$ are not necessarily matchings.

Write $a=b\bullet c$.
\end{defi}

If $m_p(\pi)=0$, this definition coincides with the one of $\pi=\alpha\circ\beta$.
For example, let $a=\{1,3,5,6,7,10\}$ and $p=4$. 
Then, we have $b=\{1,3,4,6\}$ and $c=\{1,2,3\}$.

With this new notation, the polynomials $P_{a;\beta} (z_1,\ldots,z_p,z_{\hat{p}},\ldots,z_{\hat{1}})$ are given by the following proposition:
\begin{prop}\label{prop:1step}
\[
 \Phi_{b\bullet c,-p} (z)=\Phi_{b,-p} (z^{(O)}) \sum_\beta C_{c,\beta}(\tau) \Psi_\beta (z^{(I)}).
\]
\end{prop}

\begin{proof}
Remember that the coefficients $C_{a,\pi}(\tau)$ are defined as $C_{a,\pi}(\tau)=\tau^{-d(\pi)}\Phi_a (q^\pi)$ and can be constructed using the algorithm based in a recursion formula proved in Lemma~\ref{lema:0rec}:
\[
 \Phi_{a} (q^\pi) = [s] \tau^{d(\pi)-d(\hat{\pi})} \Phi_{\hat{a}} (q^{\hat{\pi}}),
\]
where $\hat{\pi}$ is a matching obtained by removing a small arch $(j,j+1)$ from $\pi$, $\hat{a}$ is a matching obtained from $a$ by removing one element $a_i$ of the sequence such that $a_i=j$, decreasing all elements bigger then $j$ by two ($a_i\rightarrow a_i-2$ if $a_i>j$) and by one if the element is equal to $j$ ($a_i\rightarrow a_i-1$ if $a_i=j$), $s$ is the number of $a_i$ such that $a_i=j$.

Therefore we can study the polynomial $\Phi_{\at,-p} (z)$ at the special points $z=\{z_1,\ldots,z_p,q^\beta,z_{\hat{p}},\ldots,z_{\hat{1}}\}$, because this is enough to characterize the polynomial.
It is not difficult to see that we obtain exactly the same recursion formula, see Lemma~\ref{lema:prec} for the technical details.

Thus, it is not hard to prove that
\[
 \Phi_{b\bullet c,-p} (z_1,\ldots,z_p,q^\beta,z_{\hat{p}},\ldots,z_{\hat{1}}) =C_{c,\beta}(\tau) \tau^{d(\beta)} \Phi_{b,-p} (z_1,\ldots,z_p,z_{\hat{p}},\ldots,z_{\hat{1}}) ,
\]
which is the object of Corollary~\ref{cor:prec}.
This is equivalent to the result we wanted to prove.
\end{proof}

The remaining polynomial $\Phi_{b,-p}$ can be expressed by means of the polynomials $G_\alpha$: 
\begin{prop}\label{prop:2step}
\[
\Phi_{b,-p} (z^{(O)})=\sum_\alpha C_{b,\alpha}(\tau) G_\alpha (z^{(O)}).
\]
\end{prop}

\begin{proof}
If $b$ is a matching, this proposition is equivalent to the definition of $G_\alpha$.

Thus, the hard case is when $b$ is not a matching.
In that case $C_{b,\alpha}(\tau)$ is defined by the equation $\Phi_b(q^\alpha)$.
We know that $\mathcal{V}_p$ is spanned by $\Phi_f (z_1,\ldots,z_{2p})$ where $f$ is a matching of size $p$. 
So we can write $\Phi_b (z_1,\ldots,z_{2p}) = \sum_f R_{b,f}(\tau) \Phi_f (z_1,\ldots,z_{2p})$, where $R_{b,f}(\tau)$ is a matrix to be determined.
This is equivalent to 
\begin{equation}
 C_{b,\alpha}(\tau) = \sum_f R_{b,f}(\tau) C_{f,\alpha}(\tau).
\end{equation}
where $f$ and $\alpha$ are matchings, but not necessarily $b$.

We shall determine an algorithm to compute the matrix $R_{b,f}(\tau)$, but we shall only treat the case when $b_i\leq 2i-1$ for all $i$, but ignoring the condition $b_i \neq b_j$ if $i\neq j$.
Thus, the only ``anomaly'' which can occur is the existence of several $b_i$ with the same value, this is, there are some $p$ such that $\sharp\{b_i : b_i=p\}>1$.

Let $b$ be such that it has one element repeated at least twice, say $b_k=b_{k+1}=j$ and $b_{k-1}<j$, so apart from the term $(qw_k-q^{-1}w_{k+1})$ the integrand is antisymmetric on $w_k$ and $w_{k+1}$.
Using the fact that 
\begin{multline}
 \mathcal{A}\left\{ \frac{qw_k-q^{-1}w_{k+1}}{(w_k-z_j)(w_{k+1}-z_j)}+\frac{qw_k-q^{-1}w_{k+1}}{(qw_k-q^{-1}z_j)(qw_{k+1}-q^{-1}z_j)}\right.\\
+\left.\tau\frac{qw_k-q^{-1}w_{k+1}}{(qw_k-q^{-1}z_j)(w_{k+1}-z_j)}  \right\}=0
\end{multline}
we can write 
\[
 \Phi_b (z_1,\ldots,z_{2p})=-\Phi_{\tilde{b}} (z_1,\ldots,z_{2p})-\tau \Phi_{\check{b}} (z_1,\ldots,z_{2p})
\]
where $\check{b}$ is obtained from $b$ by $b_k\rightarrow b_k-1$, and $\tilde{b}$ is obtained from $b$ by $b_k\rightarrow b_k-1$ and $b_{k+1}\rightarrow b_{k+1}-1$.

If $\sharp\{b_i\text{ such that }b_i<j\}\geq j$ the integral vanishes.
Thus, we can follow this procedure until we get either a matching or it vanishes.

Now, if we look to the expression of $\Phi_{b,-p} (z_1,\ldots,z_{2p})$, we can try to apply the same procedure in order to get the same recursion.
Two things are essential, the vanishing conditions are the same, and if $b_k=b_{k+1}=j$ the integrand should be antisymmetric apart from the term $(qw_k-q^{-1}w_{k+1})$, which is true.

Having the same recursion, we can write:
\begin{align*}
 \Phi_{b,-p}(z_1,\ldots,z_{2p}) &= \sum_f R_{b,f}(\tau) \Phi_f (z_1,\ldots,z_{2p})\\
 &= \sum_f \sum_\alpha R_{b,f}(\tau) C_{f,\alpha}(\tau) G_\alpha (z_1,\ldots,z_{2p})\\
 &= \sum_\alpha C_{b,\alpha}(\tau) G_\alpha (z_1,\ldots,z_{2p})
\end{align*}
from the first to the second equations we apply the definition of $G_\alpha$.
\end{proof}

\subsection{Final details}

In conclusion we have that 
\begin{equation}
 \Phi_{b\bullet c,-p} (z) = \sum_\alpha \sum_\beta C_{b,\alpha}(\tau)C_{c,\beta}(\tau) G_\alpha (z^{(O)}) \Psi_\beta (z^{(I)}).
\end{equation}

A simple consequence of the algorithm that we use to compute the coefficients $C_{a,\pi}(\tau)$ is that it can be decomposed, as shown in Corollary~\ref{cor:0rec}: $C_{a,\alpha\circ\beta}(\tau)=C_{b,\alpha}(\tau)C_{c,\beta}(\tau)$, for $a=b\bullet c$.
Thus,
\[
 \Phi_{a,-p} (z) = \sum_\alpha \sum_\beta C_{a,\alpha\circ\beta}(\tau) G_\alpha (z^{(O)}) \Psi_\beta (z^{(I)}).
\]

When we compare with the formula
\[
 \Phi_{a,-p} (z) = \sum_\pi C_{a,\pi}(\tau) \Psi_{\pi,-p} (z),
\]
we conclude that
\begin{equation}
 \Psi_{\pi,-p} (z) = 
  \begin{cases}
   0 & \text{if }m_p(\pi)\neq 0\\
   G_\alpha (z^{(O)}) \Psi_\beta (z^{(I)})&\text{if } \pi=\alpha\circ\beta
  \end{cases}
\end{equation}
is a solution of the system of equations. 
And that solution must be unique because $C_{a,\pi}(\tau)$ is invertible.\qed

\section{Sum rule}\label{sec:sumr}

The main purpose of this section is to prove Theorem~\ref{thm:sum_G}:
\begin{align*}
 \sum_{\pi:|\pi|=n} (-1)^{d(\pi)} g_\pi &= A_n & \text{and} && \sum_{\pi:|\pi|=n} g_\pi &= (-1)^{\binom{n}{2}} \left( A_n^V \right)^2,
\end{align*}
the first one is a simple corollary of Theorem~6.16 at~\cite{tese}:
\[
 \sum_{\pi:|\pi|=n} g_\pi (-\tau) = \sum_{\pi:|\pi|=n} \psi_\pi (\tau),
\]
because it is known that $\sum_{\pi:|\pi|=n} \psi_\pi=A_n$, proved in~\cite{artic31}, and it is easy to check that $g_\pi (-1) = (-1)^{d(\pi)} g_\pi$.

\subsection{An integral formula}

The quantities $g_\pi (\tau)$ can be expressed by:
\[
 g_\pi (\tau) = \C_{\pi,a}(\tau) \phi_a (\tau,-|a|)
\]
following the definition at Section~\ref{sec:pol_qKZ}, where 
\[
 \phi_a (\tau,-|a|) = \oint \ldots \oint \prod_{i=1}^{|a|} \frac{du_i}{2\pi i u_i^{a_i}} (1+\tau u_i)^{-|a|} \prod_{j>i} (u_j-u_i) (1+\tau u_j +u_i u_j).
\]

Notice that if we change the sign of $\tau$ and at the same time the one from the variables $\{u_i\}_i$, the integral change like $	\phi_a (\tau,-|a|)=(-1)^{d(\tau)} \phi_a (\tau,-|a|)$, which in conjunction with Equation~\ref{eq:capi-tau}, we obtain: $g_\pi (-\tau) = (-1)^{d(\pi)}g_\pi (\tau)$. 

Let $\mathcal{L}_n$ be the set of matchings (in the form of sequences) of size $n$ defined by $a\in \mathcal{L}_n$ if and only if $a_i=2i-1$ or $a_i=2i-2$ for all $i$.

Following Section~3.3 of article~\cite{artic41} we can conclude that:
\begin{equation}
\sum_{\pi:|\pi|=n} g_\pi (\tau) = \sum_{a\in \mathcal{L}_n} \phi_a (\tau,-|a|).
\end{equation}
This results in:
\[
 \sum_{\pi:|\pi|=n} g_\pi (\tau) = \oint \ldots \oint \prod_i \frac{du_i}{2\pi i u_i^{2i-1}} (1+u_i) (1+\tau u_i)^{-n} \prod_{j>i} (u_j -u_i) (1+\tau u_j +u_i u_j),
\]
which can be compared with the formula:
\[
 \sum_{\pi:|\pi|=n}\psi_\pi (\tau) = \oint \ldots \oint \prod_i \frac{du_i}{2\pi i} (1+u_i) \prod_{j>i} (u_j-u_i)(1+\tau u_j +u_i u_j).
\]

In~\cite{tese} has been proved that:
\begin{multline*}
 \oint\ldots\oint \prod_i \frac{du_i}{2\pi i u_i^{2i-1}} (1+u_i) (1-\tau u_i)^{-n} \prod_{j>i} (u_j -u_i) (1-\tau u_j +u_i u_j)\\
  = \oint \ldots \oint \prod_i \frac{du_i}{2\pi i} (1+u_i) \prod_{j>i} (u_j-u_i)(1+\tau u_j +u_i u_j).
\end{multline*}
The idea of the proof, which we shall not repeat here, is that both sides count Totally Symmetric Self Complementary Plane Partitions (TSSCPP) with the same weights.
On the one hand, it was already known that the TSSCPP can be seen as lattice paths.
In this framework, we can count them using the Lindström--Gessel--Viennot formula, moreover, in this case we have a weight $\tau$ by vertical step.
This will be equivalent to the RHS.
On the other hand, we can describe the TSSCPP using a different set of lattice paths, which are called Dual Paths in~\cite{tese}.
This will give rise to the LHS.

\begin{cor}
\[
  \sum_{\pi:|\pi|=n} (-1)^{d(\pi)} g_\pi = A_n.
\]
\end{cor}

\begin{proof}

This is a consequence of the fact that $\sum_{\pi:|\pi|=n} \psi_\pi = A_n$ and $g_\pi (-1) = (-1)^{d(\pi)} g_\pi$.

\end{proof}

\subsection{Alternating Sign Matrices}

In what follows, it will be convenient to see the Fully Packed Loop as Alternating Sign Matrices.
The bijection is well known, but we shall give here a sketch of it.

An Alternating Sign Matrices (ASM) of size $n$ is a matrix $n\times n$ containing entries $\pm 1$ and $0$, such that if we ignore the zeros, the $1$ and $-1$ alternate in each column and row, and every column and row sum up to $1$.

Take a Fully Packed Loop Configuration on a $n\times n$ square lattice, in each vertex write a number $0$ if it corresponds to a corner and $\pm 1$ otherwise, and choosing the signs of the non-zero entries in such a way that we obtain a ASM.
We claim that this defines a bijection.

For example, the configuration on Figure~\ref{fig:fplexample} becomes the following ASM:
\begin{center}
 \begin{tikzpicture}[scale=0.4]
   \node[ASM] at (0,0) {0};
   \node[ASM] at (1,0) {0};
   \node[ASM] at (2,0) {0};
   \node[ASM] at (3,0) {0};
   \node[ASM] at (4,0) {1};
   \node[ASM] at (5,0) {0};
   \node[ASM] at (0,1) {0};
   \node[ASM] at (1,1) {0};
   \node[ASM] at (2,1) {0};
   \node[ASM] at (3,1) {1};
   \node[ASM] at (4,1) {0};
   \node[ASM] at (5,1) {0};
   \node[ASM] at (0,2) {0};
   \node[ASM] at (1,2) {1};
   \node[ASM] at (2,2) {0};
   \node[ASM] at (3,2) {0};
   \node[ASM] at (4,2) {-1};
   \node[ASM] at (5,2) {1};
   \node[ASM] at (0,3) {0};
   \node[ASM] at (1,3) {0};
   \node[ASM] at (2,3) {0};
   \node[ASM] at (3,3) {0};
   \node[ASM] at (4,3) {1};
   \node[ASM] at (5,3) {0};
   \node[ASM] at (0,4) {1};
   \node[ASM] at (1,4) {0};
   \node[ASM] at (2,4) {0};
   \node[ASM] at (3,4) {0};
   \node[ASM] at (4,4) {0};
   \node[ASM] at (5,4) {0};
   \node[ASM] at (0,5) {0};
   \node[ASM] at (1,5) {0};
   \node[ASM] at (2,5) {1};
   \node[ASM] at (3,5) {0};
   \node[ASM] at (4,5) {0};
   \node[ASM] at (5,5) {0};
 \end{tikzpicture} 
\end{center}

Obviously, the number of ASM of size $n$ is the famous $A_n$.
Using this transformation, we see that the vertically symmetric FPL configurations are in bijection with vertically symmetric ASM.

There is only an $1$ at the first row.
Let $A_{n,i}$ count the number of Alternating Sign Matrices with the $1$ of the first row at the $i$th column, it was proved by Zeilberger in~\cite{Zeil-ASM-ref} that:
\[
 A_{n,i} = \binom{n+i-2}{i-1} \frac{(2n-i-1)!}{(n-i)!} \prod_{j=0}^{n-2} \frac{(3j+1)!}{(n+j)!}.
\]

We know also, see for example~\cite{artic45}, that
\[
 A_n (x):= \sum_{i=1}^{n} A_{n,i} x^{i-1} = \oint \ldots \oint \prod_{i=1}^n \frac{du_i}{2\pi i u_i^{2i-1}} (1+xu_i) \prod_{j>i} (u_j-u_i) (1+u_j+u_i u_j).
\]
In fact, this was conjectured in Zinn-Justin and Di Francesco's article~\cite{artic41}, in the same article the authors reformulate this conjecture in a different equation which was proved by Zeilberger in~\cite{Zeil-qKZ}.

Thus, it is straightforward to see that $\sum_\pi g_\pi = \sum_\pi \psi_\pi (-1) = (-1)^{\binom{n}{2}} A_n (-1)$.

\subsection{The $-1$ enumeration of ASM}\label{sec:-enum}

Next, we prove that the $-1$ enumeration of Alternating Sign Matrices $A_n (-1)$ is exactly the number of Vertically Symmetric Alternating Sign Matrices $A^V_n$ squared.
This result is already present in Di Francesco~\cite[Equation~2.8]{DF-qKZ-TSSCPP} and a detailed proof can be found in Williams' article~\cite{Nathan}.
For the sake of completeness we shall prove it in detail.

\begin{prop}
 We want to prove that $A_n(-1) = (A^V_n)^2$, \ie 
 \begin{multline*}
 \sum_i (-1)^{i-1} \binom{n+i-2}{i-1} \frac{(2n-i-1)!}{(n-i)!} \prod_{j=0}^{n-2} \frac{(3j+1)!}{(n+j)!}\\
= \begin{cases}
0 & \text{if }n\text{ is even};\\
\left(\frac{1}{2^m} \prod_{i=1}^m \frac{(6i-2)!(2i-1)!}{(4i-1)!(4i-2)!}\right)^2 & \text{if }n=2m+1.
 \end{cases}
\end{multline*}
\end{prop}

\begin{proof}
We can rewrite the expression:
\begin{align*}
 A_n (-1)&= \left.(-x)^{i-1}\binom{n+i-2}{i-1} \frac{(2n-i-1)!}{(n-i)!} x^{n-i} \prod_{j=0}^{n-2} \frac{(3j+1)!}{(n+j)!}\right|_{x=1}\\
  &=\left.\sum_{i=1}^n\sum_{k=0}^{n-1} (-x)^{i-1}\binom{n+i-2}{i-1} \frac{(n+k-1)!}{k!} x^{k} \prod_{j=0}^{n-2} \frac{(3j+1)!}{(n+j)!}\right|_{x^{n-1}},
\end{align*}
where the subscript $x^{n-1}$ means that we select the coefficient of $x^{n-1}$.
Changing $i-1 \rightarrow i$ and changing the limits, because they do not interfere in our computation:
\[
 A_n (-1)= \left.\sum_{i=0}^\infty\sum_{k=0}^\infty (-x)^i\binom{n+i-1}{i} \binom{n+k-1}{k} x^{k} (n-1)!\prod_{j=0}^{n-2} \frac{(3j+1)!}{(n+j)!}\right|_{x^{n-1}}.
\]

But 
\[
 \frac{1}{(1+x)^n} = \sum_{i=0}^\infty (-x)^i \binom{n+i-1}{i}.
\]
Applying this, we get:
\begin{align*}
  A_n (-1) &= \left.\frac{1}{(1+x)^n}\frac{1}{(1-x)^n} (n-1)!\prod_{j=0}^{n-2} \frac{(3j+1)!}{(n+j)!}\right|_{x^{n-1}}\\
  &=\left.\frac{1}{(1-x^2)^n} (n-1)!\prod_{j=0}^{n-2} \frac{(3j+1)!}{(n+j)!}\right|_{x^{n-1}}\\
  &=\left.\sum_i x^{2i} \binom{n+i-1}{i} (n-1)!\prod_{j=0}^{n-2} \frac{(3j+1)!}{(n+j)!}\right|_{x^{n-1}}.
\end{align*}
And this is zero if $n-1$ is odd. Set $n=2m+1$ with $m$ integer, thus:
\[
 A_{2m+1}(-1) = \binom{3m}{m} (2m)! \prod_{j=0}^{2m-1} \frac{(3j+1)!}{(2m+j+1)!}.
\]

A simple manipulation shows that this is equal to $\left(A_{2m+1}^V\right)^2$.
\end{proof}

\section{Further Questions}

\subsection{Solving the conjectures}

This paper is, in a certain way, the continuation of the article~\cite{negative}. 
Although, we solve here two conjectures, there still two more: the fact that that all coefficients of $A_\pi (t)$ are positive (in fact, we have some numerical evidence that all roots of $A_\pi (t)$ have negative real part) and the multiplicity of the real roots.

Also, the value of $g_{()^{2m+1}}$ it is still a conjecture.
In~\cite{tese} the author presents an interpretation of this value as counting a certain subset $\mathcal{R}$ of the Totally Symmetric Plane Partitions.
In fact it is not hard to prove that this subset $\mathcal{R}$ is exactly the subset $\mathcal{P}_n^R$ which appears in Ishikawa's article~\cite{Masao-tsscpp1}.
Therefore, this conjecture it is equivalent to the unweighted version of \cite[Conjecture~4.2]{Masao-tsscpp1}.

\subsection{Combinatorial reciprocity}

We recover here one of the ideas of~\cite{negative}.
The Theorems proved here, together with the conjectures of~\cite{negative} which remain unproved suggest that $g_\pi$ and $A_\pi (-t)$ for $t\in\mathbb{N}$ have a combinatorial interpretation.
That is, we believe that exist yet-to-be-discovered combinatorial objects indexed by the matchings $\pi$ counted by $|g_\pi|$ or by $|A_\pi (-p)|$.

Notice that, even if the sum rule of $A_\pi$ and $g_\pi$ are related ($\sum_\pi (-1)^{d(\pi) g_\pi = \sum_\pi A_\pi}$), both quantities are essentially different, it is, they have different symmetries.
On the one hand, the $A_\pi$ are stable in respect of the rotation and mirror symmetries. 
On the other hand, the $g_\pi$ are stable by inclusion, this is $g_{(\pi)}=g_\pi$.

The well-known \emph{Ehrhart polynomial} $i_P(t)$ of a lattice polytopes $P$, which counts the number of lattice points in $tP$ when $t$ is a positive integer, has an interesting property.
When $t$ is negative, $(-1)^{\dim P} i_P(t)$ counts the lattice points strictly inside of $-tP$ (see~\cite{BeckRobins} for instance).
We believe that should be something similar in the quantities $A_\pi (t)$.

\subsection{A new vector space of polynomials}

The polynomial $\Psi_\pi (z_1,\ldots,z_{2n})$ can be seen as the solution of the quantum Knizhnik--Zamolodchikov equation, moreover they define a vector subspace of polynomials characterized by a vanish condition (the \emph{wheel condition}) and a overall degree.
These polynomials, are also related to the non-symmetric Macdonald polynomials and can be constructed using some difference operators as showed by Lascoux, de Gier and Sorrell in~\cite{Lascoux-KL-M}.

In the same way, $G_\pi (z_1,\ldots,z_{2n})$ span a vector subspace of polynomials with a certain fixed degree.
Therefore, it will be interesting to fully characterize this vector subspace, and see if there is some other way to construct them, as the difference operator used in~\cite{Lascoux-KL-M}.

\section*{Acknowledgments}
The author is thankful to Ferenc Balogh for all the interesting discussions about this subjects and others.
The author would like also to thank Philippe Nadeau, with whom he published the article that inspiredu this one.

\appendix
\section{Proof of some technical lemmas}

\subsection{Polynomiality}\label{sec:poly}

In this section we prove that the quantities $\Phi_{a,-p} (z)$ are polynomials. 
Using Cauchy's integral formula, the integral is not so complicated, this is, it is enough to chose poles $(w_k-z_j)^{-1}$ for all $k$ such that $j\leq a_k$ (and no repeated $j$), and compute the residues on these poles.
Then we sum over all possible choices.
It is then clear that the result is a sum of ration of polynomials, and we want to prove that this sum has no poles.

Looking to the integral formula we can identify three sources of problems:
\begin{enumerate}
\item The term $\prod_i \prod_{j\leq a_i} (w_i-z_j)^{-1}$, which can originate poles like $(z_k - z_j)^{-1}$ with $k\neq j$ and $j,k\leq a_i$ for some $i$;
\item The term $\prod_i \prod_{j>a_i} (q w_i -q^{-1} z_j)^{-1}$, which can originate poles like $(q z_k -q^{-1} z_j)^{-1}$ with $k\leq a_i <j$ for some $i$;
\item The term $\prod_i \prod_{j=1}^p (q z_j -q^{-1} w_i)^{-1}$, which can originate poles like $(qz_j - q^{-1} z_k)^{-1}$ with $j\leq p$ and $k\leq a_i$ for some $i$. 
\end{enumerate}

The proof follows always the same path, we isolate a certain pole, and then we find that its coefficient vanishes.

\subsubsection{Poles like $(z_k-z_j)^{-1}$}

In order to obtain this pole we need to chose the residue at $(w_i-z_j)^{-1}$ with $k\leq a_i$ and/or the residue at $(w_l-z_k)^{-1}$ with $j\leq a_l$.

Let us assume that we chose both of them.
In the sum, there is an other term coming from the opposite choice $(w_l-z_j)^{-1}$ and $(w_i-z_k)^{-1}$ because $j,k\leq a_i,a_l$. 
Notice that there is a term $(w_i-w_l)$ which is $(z_j-z_k)$ in the first case and $(z_k-z_j)$, so we can write the sum of these two terms:
\[
 \frac{f(z_j,z_k,w_i=z_j,w_l=z_k)(z_j-z_k)+f(z_j,z_k,w_i=z_k,w_l=z_j)(z_k-z_j)}{(z_j-z_k)(z_k-z_j)}
\]
where $f$ is some analytical function on the point $z_k=z_j$.
So the pole disappears.

In the case that we chose only one of the poles, the analyse is the same. 
We chose $(w_i-z_j)^{-1}$ and it is easy to verify that the pole will cancel with the one coming from choosing the pole $(w_i-z_k)^{-1}$.

\subsubsection{Poles like $(qz_j-q^{-1}z_k)$}
When $k>\hat{p}$, the term is compensated by the zero of $\prod_{j}^p (qw_i-q^{-1}z_{\hat{j}})$.

The poles from $\prod_{i}\prod_{j=1}^p( qz_j-q^{-1}w_i)^{-1}$ will be cancelled by the term $\prod_{1\leq i,j \leq p} (qz_i-q^{-1}z_j) \prod_{2\leq i,j\leq p} (qz_i - q^{-1}z_{\hat{j}}) \prod_i^{p} \prod_{p<j<\hat{p}} (qz_i-q^{-1}z_j)$ except for some terms containing $z_1$ which can be ruled out by picking the pole $(w_1-z_1)^{-1}$.

Finally, the only poles that can appear are coming from the term $(qw_i-q^{-1}z_j)^{-1}$ with $j<\hat{p}$.
If $j>p$, the pole correspond to $w_i=z_k$ with $k\leq a_i<j$, so $k<j$ and this is ruled out by the term $\prod_{p<k<j<\hat{p}} (qz_k-q^{-1}z_j)$.
If $j\leq p$, such term doesn't exist. If we do not pick any pole $(w_l-z_j)^{-1}$ we can use the term $\prod_{1\leq i,j\leq p}(qz_i-q^{-1}z_j)$, if we pick $(w_l-z_j)^{-1}$ so we have that $k\leq a_i<j\leq a_l$, so $k<l$ and we have the term $(qw_i-q^{-1}w_l)=(qz_k-q^{-1}z_j)$. \qed

\subsection{An anti-symmetrization formula}
We anti-symmetrize the expression\linebreak $\prod_{i<j}^s (qw_i-q^{-1}w_j)$.
This will be useful for some Lemmas~\ref{lema:0rec} and~\ref{lema:prec}.
The result is not complicated and can be obtained by several ways.
 
Define $\mathcal{A}$ as being the anti-symmetrization of a function: $\mathcal{A} \left(f(w_1,\ldots,w_k) \right) =\frac{1}{k!} \sum_{\sigma\in S_k} \text{sign}(\sigma) f(w_{\sigma_1},\ldots,w_{\sigma_k})$.

\begin{lemma}\label{lema:anti} 
\[
 \mathcal{A}\left(\prod_{i<j}^k (qw_i-q^{-1}w_j)\right)=\frac{[k]!}{k!} \prod_{i<j}^k (w_i-w_j),
\] 
where $[k]!=[k][k-1]\ldots[1]$, and $[k]:=\frac{q^k-q^{-k}}{q-q^{-1}}=q^{k-1}+q^{k-3}+\ldots+q^{-k+1}$.
\end{lemma}

\begin{proof}
It is obvious that the result is a homogeneous polynomial of degree $\binom{k}{2}$, and as it is antisymmetric on the $\{w_i\}_{1\leq i\leq k}$ it must be a multiple of $\Delta_k := \prod_{i<j}^k (w_i-w_j)$.
The only difficulty is to find the coefficient.

Let $q_{i,j}:=(qw_i-q^{-1}w_j)$.
On one hand we have:
\begin{align*}
 \left.\mathcal{A}\left( \prod_{i<j}^k (qw_i-q^{-1}w_j)\right)\right|_{w_k=0} &= c_k \left.\Delta_k\right|_{w_k=0}\\
 &=c_k \left(\prod_{i=1}^{k-1} w_i \right) \Delta_{k-1}.
\end{align*}

On the other hand, we have:
\begin{align*}
 \left.\mathcal{A}\left( \prod_{i<j}^k q_{i,j}\right)\right|_{w_k=0} &= \frac{1}{k!}\left.\sum_{\sigma\in S_k} \text{sign}(\sigma) \prod_{i<j} q_{\sigma_i,\sigma_j}\right|_{w_k=0} \\
 &=\frac{1}{k!}\left.\sum_{r=1}^k \sum_{\substack{\sigma\in S_k \\ \sigma_r=k}} \text{sign}(\sigma) \prod_{i<r}q_{\sigma_i,k}\prod_{j>r}q_{k,\sigma_j}\prod_{\substack{i<j\\i,j\neq r}}q_{\sigma_i,\sigma_j}\right|_{w_k=0} \\
 &= \frac{1}{k!} \sum_{r=1}^k \sum_{\substack{\sigma\in S_k\\\sigma_r=k}}\text{sign}(\sigma)  (-1)^{k-r} q^{2r-k-1}\prod_{i=1}^{k-1} w_i \prod_{\substack{i<j\\i,j\neq r}} q_{\sigma_i,\sigma_j}\\
 &= \left(\sum_r^k q^{2r-k-1}\right)\frac{1}{k!} \left(\prod_{i=1}^{k-1} w_i \right) \sum_{\epsilon \in S_{k-1}} \text{sign}(\epsilon) \prod_{i<j}^{k-1}q_{\epsilon_i,\epsilon_j}\\
 &= \frac{[k]}{k} c_{k-1} \Delta_{k-1}.
\end{align*}

So we have $c_k=\frac{[k]}{k} c_{k-1}$.
We easily check that $c_1=1$, the result follows.
\end{proof}

\subsection{The recursion formulæ}\label{sec:rec}

In this section we prove the following lemma:

\begin{lemma}\label{lema:0rec}
 Let $\bi$ be a matching of size $n$ let $p$ be a integer such that $0<p<n$, and such that $m_p (\bi)=0$, this means that we can write $\bi=\bt$ with $\bj$ and $\bk$ two matchings of size, respectively, $p$ and $r:=n-p$.
 Let $\ai$ be a second matching, represented by a sequence. 
 Write $\ai=\at$, such that $\aj$ and $\ak$ are the outer and inner part, respectively, and are not necessarily matchings.

 Clearly $\bk$ has a small arch, say $(\zj,\zj+1)$.
 Let $\hat{\bk}$ be the matching obtained from $\bk$ by removing the small arch $(\zj,\zj+1)$.
 Let $s$ count the $\ak_i$ such that $\ak_i=\zj$. 
 Let $\hat{\ak}$ be the sequence obtained from $\ak$ by keeping all $\ak_i<\zj$, the $s$ $\ak_i=\zj$ are transformed in $(s-1)$ $\zj-1$, and all the others $\ak_i>\zj$ are transformed $\ak_i \rightarrow \ak_i-2$.
  
 We claim that 
 \[
  \Phi_\at (q^{\bt}) = [s] \tau^{d(\bt)-d(\bj\circ\hat{\bk})} \Phi_{\aj\bullet\hat{\ak}}(q^{\bj\circ\hat{\bk}}).
 \]
\end{lemma}

The proof of this lemma is rather long but straightforward. 
The proof can be found in the literature, we shall reproduce it here because the understanding of the proof is important to another result.

\begin{proof}
 Notice that  when $z_{\zj+1}=q^2 z_\zj$ the pre-factor vanishes, thus we need to compensate it by choosing a pole $(w_\wj-z_\zj)^{-1}$ which will make appear the pole $(qz_\zj-q^{-1}z_{\zj+1})^{-1}$, so we must have $\zj\leq \ai_\wj < \zj+1$, which means $\ai_\wj = \zj$.
 Assume that there are $s\geq 1$ such $\ai_\wj$.

 The idea of the proof, is to pick all $s$ different poles with this property, and we will have a lot of cancellations, turning possible an identification with a smaller integral $\Phi_{\aj\bullet\hat{\ak}} (q^{\bj\circ\hat{\bk}})$.

 Because our formula is rather big, we will use the following shortcuts:
\begin{itemize}
  \item We divide the $z$ variables into two regions, and give different indices depending on the region  $\zi <\zj$, $\zk>\zj+1$;

  \item For the variables in the integral, we divide in four regions 
 (recall that we have $s$ $\ai_\wj = \zj$, and let $(w_\wj-z_\zj)^{-1}$ be the chosen pole), $\wh$ is such that $\ai_\wh <\zj$, $\wi<\wj$ such that $\ai_\wi=\zj$, $\wk>\wj$ such that $\ai_\wk=\zj$ and $\wl$ is such that $\ai_\wl > \zj$;

  \item We will use the notation $\q{\zi}{\zk}:=\prod_{\zi,\zk}(qz_\zi-q^{-1}z_\zk)$. 
  And equivalently, for any variable.
  For example $\q{\zj,\wl}=\prod_{\wl} (qz_\zj-q^{-1}w_\wl)$.
  Notice that we follow the rule Latin letters for the variables $z$ and Greek letters for the variables $w$;

  \item We use also the notation $\um{\zi}{\wk}$, meaning that we replace $q$ by $1$;

  \item The symbol $'$ means smaller, for example $\um{\zi}{\zi'}=\prod_{\zi>\zi'} (z_\zi-z_{\zi'})$;

  \item In order to keep simplicity, we omit every term which depends on none of the following variables $\{z_\zj,z_{\zj+1},w_\wi,w_\wj,w_\wk\}$, we omit even all integration symbols except the one of $w_\wj$;

  \item Define $n_\zi$ as the number of $z_\zi$, and equivalently for all other variables;

  \item Finally $\xi=(q-q^{-1})$. 
\end{itemize}

Let us rewrite our equation with this notation:
\begin{align*}
 \Phi_\at (q^\bt) &= \sum_{\wj}^s \xi^{-n(n-1)} \q{\zi}{\zj} \q{\zi}{\zj+1} \q{\zj}{\zj+1} \q{\zj}{\zk} \q{\zj+1}{\zk} \oint \frac{dw_\wj}{2\pi i} \frac{\um{\wi}{\wh}\um{\wj}{\wh}\um{\wk}{\wh}\um{\wi}{\wi'}\um{\wj}{\wi}\um{\wk}{\wi}\um{\wl}{\wi}}{\q{\wh}{\zj}\q{\wh}{\zj+1}\um{\wi}{\zi}\um{\wi}{\zj}\q{\wi}{\zj+1}\q{\wi}{\zk}}\\
 &\quad \times \frac{\um{\wk}{\wj}\um{\wl}{\wj}\um{\wk}{\wk'}\um{\wl}{\wk}\q{\wh}{\wi}\q{\wh}{\wj}\q{\wh}{\wk}\q{\wi'}{\wi}\q{\wi}{\wj}\q{\wi}{\wk}\q{\wi}{\wl}\q{\wj}{\wk}\q{\wj}{\wl}\q{\wk'}{\wk}\q{\wk}{\wl}}{\um{\wj}{\zi}\um{\wj}{\zj}\q{\wj}{\zj+1}\q{\wj}{\zk}\um{\wk}{\zi}\um{\wk}{\zj}\q{\wk}{\zj+1}\q{\wk}{\zk}\um{\wl}{\zj}\um{\wl}{\zj+1}}.
\end{align*} 

We use Lemma~\ref{lema:anti}:
\begin{align*}
 \Phi_\at (q^\bt) &= \frac{[s]!}{s!} \sum_{\wj}^s \xi^{-n(n-1)} \q{\zi}{\zj} \q{\zi}{\zj+1} \q{\zj}{\zj+1} \q{\zj}{\zk} \q{\zj+1}{\zk} \oint \frac{dw_\wj}{2\pi i} \frac{\um{\wi}{\wh}\um{\wj}{\wh}\um{\wk}{\wh}\um{\wi}{\wi'}\um{\wj}{\wi}\um{\wk}{\wi}\um{\wl}{\wi}}{\q{\wh}{\zj}\q{\wh}{\zj+1}\um{\wi}{\zi}\um{\wi}{\zj}\q{\wi}{\zj+1}\q{\wi}{\zk}}\\
 &\quad \times \frac{\um{\wk}{\wj}\um{\wl}{\wj}\um{\wk}{\wk'}\um{\wl}{\wk}\q{\wh}{\wi}\q{\wh}{\wj}\q{\wh}{\wk}\um{\wi'}{\wi}\um{\wi}{\wj}\um{\wi}{\wk}\q{\wi}{\wl}\um{\wj}{\wk}\q{\wj}{\wl}\um{\wk'}{\wk}\q{\wk}{\wl}}{\um{\wj}{\zi}\um{\wj}{\zj}\q{\wj}{\zj+1}\q{\wj}{\zk}\um{\wk}{\zi}\um{\wk}{\zj}\q{\wk}{\zj+1}\q{\wk}{\zk}\um{\wl}{\zj}\um{\wl}{\zj+1}}.
\end{align*} 

And finally we integrate around $w_\wj = z_\zj$:
\begin{align*}
 \Phi_\at (q^\bt) &= \frac{[s]!}{s!} \sum_{\wj}^s \xi^{-n(n-1)} \q{\zi}{\zj} \q{\zi}{\zj+1} \q{\zj}{\zj+1} \q{\zj}{\zk} \q{\zj+1}{\zk} \frac{\um{\wi}{\wh}\um{\zj}{\wh}\um{\wk}{\wh}\um{\wi}{\wi'}\um{\zj}{\wi}\um{\wk}{\wi}\um{\wl}{\wi}}{\q{\wh}{\zj}\q{\wh}{\zj+1}\um{\wi}{\zi}\um{\wi}{\zj}\q{\wi}{\zj+1}\q{\wi}{\zk}}\\
 &\quad \times \frac{\um{\wk}{\zj}\um{\wl}{\zj}\um{\wk}{\wk'}\um{\wl}{\wk}\q{\wh}{\wi}\q{\wh}{\zj}\q{\wh}{\wk}\um{\wi'}{\wi}\um{\wi}{\zj}\um{\wi}{\wk}\q{\wi}{\wl}\um{\zj}{\wk}\q{\zj}{\wl}\um{\wk'}{\wk}\q{\wk}{\wl}}{\um{\zj}{\zi}\q{\zj}{\zj+1}\q{\zj}{\zk}\um{\wk}{\zi}\um{\wk}{\zj}\q{\wk}{\zj+1}\q{\wk}{\zk}\um{\wl}{\zj}\um{\wl}{\zj+1}}\\
  &= \frac{[s]!}{s!} \sum_{\wj}^s \xi^{-n(n-1)} \q{\zi}{\zj} \q{\zi}{\zj+1} \q{\zj+1}{\zk} \frac{\um{\wi}{\wh}\um{\zj}{\wh}\um{\wk}{\wh}\um{\wi}{\wi'}\um{\zj}{\wi}\um{\wk}{\wi}\um{\wl}{\wi}}{\q{\wh}{\zj+1}\um{\wi}{\zi}\q{\wi}{\zj+1}\q{\wi}{\zk}}\\
 &\quad \times \frac{\um{\wk}{\wk'}\um{\wl}{\wk}\q{\wh}{\wi}\q{\wh}{\wk}\um{\wi'}{\wi}\um{\wi}{\wk}\q{\wi}{\wl}\um{\zj}{\wk}\q{\zj}{\wl}\um{\wk'}{\wk}\q{\wk}{\wl}}{\um{\zj}{\zi}\um{\wk}{\zi}\q{\wk}{\zj+1}\q{\wk}{\zk}\um{\wl}{\zj+1}}.
\end{align*} 

Using the fact that $z_\zj=q^{-1}$ and $z_{\zj+1}=q$ the following equalities are straightforward:
\begin{align*}
  \frac{\q{\zi}{\zj+1}}{\um{\zj}{\zi}}&=(-q)^{n_\zi}; &
  \frac{\um{\zj}{\wi}}{\q{\wi}{\zj+1}}&=(-q)^{-n_\wi}; &
  \frac{\um{\zj}{\wh}}{\q{\wh}{\zj+1}}&=(-q)^{-n_\wh};\\
  \frac{\um{\zj}{\wk}}{\q{\wk}{\zj+1}}&=(-q)^{-n_\wk};&
  \frac{\q{\zj}{\wl}}{\um{\wl}{\zj+1}}&=(-q)^{-n_\wl}.
\end{align*}
We obtain,
\begin{align*}
 \Phi_\at (q^\bt) &=(-q)^{n_\zi-n_\wh-n_\wi-n_\wk-n_\wl} \frac{[s]!}{s!} \sum_{\wj}^s \xi^{-n(n-1)} \q{\zi}{\zj} \q{\zj+1}{\zk} \frac{\um{\wi}{\wh}\um{\wk}{\wh}\um{\wi}{\wi'}\um{\wk}{\wi}\um{\wl}{\wi}}{\um{\wi}{\zi}\q{\wi}{\zk}\um{\wk}{\zi}\q{\wk}{\zk}}\\
 &\quad \times \um{\wk}{\wk'}\um{\wl}{\wk}\q{\wh}{\wi}\q{\wh}{\wk}\um{\wi'}{\wi}\um{\wi}{\wk}\q{\wi}{\wl}\um{\wk'}{\wk}\q{\wk}{\wl}\\
 &=(-q)^{n_\zi-(n-1)} \frac{[s]}{s} \sum_{\wj}^s \xi^{-n(n-1)} \q{\zi}{\zj} \q{\zj+1}{\zk} \frac{\um{\wi}{\wh}\um{\wk}{\wh}\um{\wi}{\wi'}\um{\wk}{\wi}\um{\wl}{\wi}\um{\wk}{\wk'}\um{\wl}{\wk}}{\um{\wi}{\zi}\q{\wi}{\zk}\um{\wk}{\zi}\q{\wk}{\zk}}\\
 &\quad \times \q{\wh}{\wi}\q{\wh}{\wk}\q{\wi'}{\wi}\q{\wi}{\wk}\q{\wi}{\wl}\q{\wk'}{\wk}\q{\wk}{\wl}.
\end{align*} 
This is exactly the $\Phi_{\aj\bullet\hat{\ak}} (q^{\bj\circ\hat{\bk}})$:
\begin{align*}
 \Phi_\at (q^\bt) &=(-q)^{n_\zi-(n-1)} [s] \xi^{-2(n-1)} \q{\zi}{\zj} \q{\zj+1}{\zk} \Phi_{\aj\bullet\hat{\ak}} (q^{\bj\circ\hat{\bk}})\\
 &=(-q)^{-(n-1)} [s] \xi^{-2(n-1)} \prod_{\zi}(q^{-1}-q^2 z_\zi) \prod_\zk (q^2-q^{-1}z_\zk) \Phi_{\aj\bullet\hat{\ak}} (q^{\bj\circ\hat{\bk}}).
\end{align*}
Finally, if we replace each $z_\zi = q^{\pm 1}$ and $z_\zj = q^{\pm 1}$, we get that:
\begin{align*}
 (q^{-1}-q^2 z_\zi) &= 
  \begin{cases}
   -\xi & \text{if } z_\zi=q^{-1}\\
   q \tau \xi & \text{if } z_\zi=q
  \end{cases} ;&
 (q^2-q^{-1} z_\zk) &= 
  \begin{cases}
   - \tau \xi & \text{if } z_\zk=q^{-1}\\
   q \xi & \text{if } z_\zk=q
  \end{cases} 
\end{align*}

Is a simple exercise to check that all factors $q$, $(-1)$ and $\xi$ cancel.
We can also check that the number of $z_\zi=q$ plus the number of $z_\zk=q^{-1}$ is exactly $d(\bt)-d(\bj\circ\hat{\bk})$, proving the lemma. 

\end{proof}

We have seen that the $\Psi_\bi (z)$ form a basis of the vector space $\mathcal{V}_n$. 
$\Phi_\ai(z_1,\ldots,z_{2n})$ lives at the vector space $\mathcal{V}_n$ too, so that we can write as $\Phi_\ai (z) = \sum_\bi C_{\ai,\bi}(\tau) \Psi_\bi (z)$.

And this coefficients can be identified using the relation $\Phi_{\ai} (q^{\bi}) = C_{\ai,\bi}(\tau) \tau^{d(\bi)}$. 
Thus, this lemma is telling us how we construct these $C_{\ai,\bi}(\tau)$: we pick a small arch at $\bi$, we apply the recursion and we get smaller matchings $\hat{\ai}$ and $\hat{\bi}$, and the whole thing is multiplied by $[s] \tau^{d(\bi)-d(\hat{\bi})}$.
Thus, the coefficient $C_{\ai,\bi}(\tau)$ it will be the multiplication by $[s]$ at each step, because the power on $\tau$ it will be exactly $\tau^{d(\ai)}$. 

Following this procedure, we can see what happen in the case that $\bi=\bt$:
\begin{cor}\label{cor:0rec}
 The coefficients $C_{\at,\bt}(\tau)$ obey to the simple decomposition 
\[
 C_{\at,\bt}(\tau) = C_{\aj,\bj}(\tau) C_{\ak,\bk}(\tau).
\]
\end{cor}

\begin{proof}

On one hand, if we apply the algorithm $n$ times, we get:
\[
 \Phi_\at (q^\bt) = \tau^{d(\bt)}C_{\at,\bt}(\tau) .
\]

On the other hand, we can do the recursion starting with the $r$ inner variables, but the algorithm is independent of all details, so:
\begin{align*}
 \Phi_\at (q^\bt) &= \tau^{d(\bt)-d(\bj)}C_{\ak,\bk}(\tau)  \Phi_\aj (q^\bj)\\
                  &= \tau^{d(\bt)}C_{\ak,\bk}(\tau)  C_{\aj,\bj}(\tau).
\end{align*}

Notice that $\aj$ is not necessarily a matching, but nevertheless the coefficient $C_{\aj,\bj}(\tau)$ is defined by $\Psi_{\aj}(q^{\bj})$ and similarly for $\ak$.

This completes the proof.
\end{proof}

Now we repeat the procedure for the case $\Phi_{\ai,-p} (z)$.
The point is that, this polynomial is proportional to $\Psi_\bk (z^{(I)})$, where the $\bk$ are of size $r$.
We then check that the algorithm of Lemma~\ref{lema:0rec} remains essentially the same, \ie 
\begin{lemma}\label{lema:prec}
 Let $\bi$ and $\ai$ be two matching as in Lemma~\ref{lema:0rec}, thus $\bi=\bt$ and $\ai=\at$.
 Let $(j,j+1)$ be a small arch on $\bk$. 
 $\hat{\bk}$, $\hat{\ak}$ and $s$ are defined as in Lemma~\ref{lema:0rec}.

 We claim that
\[
 \Phi_{\at,-p} (z_\zh,q^\bk,z_\zl) = [s] \tau^{d(\bk)-d(\hat{\bk})} \Phi_{\aj\bullet\hat{\ak},-p} (z_\zh,q^{\hat{\bk}},z_\zl).
\]
\end{lemma}  

The proof will follow exactly the same steps as the one for Lemma~\ref{lema:0rec}, though there are some details that differ.
We should repeat it here for sake of completeness.

\begin{proof}
 We use all definitions made on the proof of Lemma~\ref{lema:0rec}, with only one difference:
\begin{itemize}
 \item We divide the $z$ variables in four regions: $\zh \leq p$, $p<\zi<\zj$, $\zj+1<\zk<\hat{p}$ and $\zl\geq \hat{p}$.
\end{itemize}

We rewrite the expression, using the new notation and applying Lemma~\ref{lema:anti}:
\begin{align*}
 \Phi_{\at,-p} (z_\zh,q^\bk,z_\zl) &= \frac{[s]!}{s!} \sum_{\wj}^s \xi^{-r(r-1)-(p-1)^2} \q{\zh}{\zj} \q{\zh}{\zj+1} \q{\zi}{\zj} \q{\zi}{\zj+1} \q{\zj}{\zj+1} \q{\zj}{\zk} \q{\zj+1}{\zk} \\
   &\quad\times\oint \frac{dw_\wj}{2\pi i} \frac{\um{\wl}{\wk}\um{\wl}{\wj}\um{\wl}{\wi}\um{\wk}{\wk'}\um{\wk}{\wj}\um{\wk}{\wi}\um{\wk}{\wh}\um{\wj}{\wi}\um{\wj}{\wh}\um{\wi}{\wi'}\um{\wi}{\wh}}{\q{\wh}{\zj}\q{\wh}{\zj+1}\um{\wi}{\zh}\um{\wi}{\zi}\um{\wi}{\zj}\q{\wi}{\zj+1}\q{\wi}{\zk}\q{\wi}{\zk}\um{\wj}{\zh}\um{\wj}{\zi}\um{\wj}{\zj}}\\ 
   &\quad\times\frac{\q{\wk}{\wl}\q{\wj}{\wl}\q{\wi}{\wl}\um{\wk'}{\wk}\um{\wj}{\wk}\um{\wi}{\wk}\q{\wh}{\wk}\um{\wi}{\wj}\q{\wh}{\wj}\um{\wi'}{\wi}\q{\wh}{\wi}}{\q{\wj}{\zj+1}\q{\wj}{\zk}\q{\wj}{\zl}\um{\wk}{\zh}\um{\wk}{\zi}\um{\wk}{\zj}\q{\wk}{\zj+1}\q{\wk}{\zk}\q{\wk}{\zl}\um{\wl}{\zj}\um{\wl}{\zj+1}}\frac{\q{\wi}{\zl}\q{\wj}{\zl}\q{\wk}{\zl}}{\q{\zh}{\wi}\q{\zh}{\wj}\q{\zh}{\wk}}.
\end{align*} 

We integrate around $w_\wj=z_\zj$:
\begin{align*}
 \Phi_{\at,-p} (z_\zh,q^\bk,z_\zl) &= \frac{[s]!}{s!} \sum_{\wj}^s \xi^{-r(r-1)-(p-1)^2} \q{\zh}{\zj+1} \q{\zi}{\zj} \q{\zi}{\zj+1} \q{\zj+1}{\zk} \\
   &\quad\times \frac{\um{\wl}{\wk}\um{\wl}{\wi}\um{\wk}{\wk'}\um{\wk}{\wi}\um{\wk}{\wh}\um{\zj}{\wi}\um{\zj}{\wh}\um{\wi}{\wi'}\um{\wi}{\wh}}{\q{\wh}{\zj+1}\um{\wi}{\zh}\um{\wi}{\zi}\q{\wi}{\zj+1}\q{\wi}{\zk}\q{\wi}{\zk}\um{\zj}{\zh}\um{\zj}{\zi}}\\ 
   &\quad\times\frac{\q{\wk}{\wl}\q{\zj}{\wl}\q{\wi}{\wl}\um{\wk'}{\wk}\um{\zj}{\wk}\um{\wi}{\wk}\q{\wh}{\wk}\um{\wi'}{\wi}\q{\wh}{\wi}}{\um{\wk}{\zh}\um{\wk}{\zi}\q{\wk}{\zj+1}\q{\wk}{\zk}\q{\wk}{\zl}\um{\wl}{\zj+1}}\frac{\q{\wi}{\zl}\q{\wk}{\zl}}{\q{\zh}{\wi}\q{\zh}{\wk}}\\
   &= (-q)^{n_\zh+n_\zi} (-q)^{-(n-1)} [s] \xi^{-r(r-1)-(p-1)^2} \q{\zi}{\zj} \q{\zj+1}{\zk} \\
   &\quad\times \frac{\um{\wl}{\wk}\um{\wl}{\wi}\um{\wk}{\wk'}\um{\wk}{\wi}\um{\wk}{\wh}\um{\wi}{\wi'}\um{\wi}{\wh}}{\um{\wi}{\zh}\um{\wi}{\zi}\q{\wi}{\zk}\q{\wi}{\zk}}
   \frac{\q{\wk}{\wl}\q{\wi}{\wl}\q{\wk'}{\wk}\q{\wi}{\wk}\q{\wh}{\wk}\q{\wi'}{\wi}\q{\wh}{\wi}}{\um{\wk}{\zh}\um{\wk}{\zi}\q{\wk}{\zk}\q{\wk}{\zl}}\frac{\q{\wi}{\zl}\q{\wk}{\zl}}{\q{\zh}{\wi}\q{\zh}{\wk}},
\end{align*}
but $n_\zh=p$.

Now we can identify $\Phi_{\aj\bullet\hat{\ak},-p} (z_\zh,q^{\hat{\bk}},z_\zl)$: 
\begin{align*}
 \Phi_{\at,-p} (z_\zh,q^\bk,z_\zl) &= (-q)^{-(r-1)} [s] \xi^{-2(r-1)} \prod_\zi (q^{-1} -q^2 z_\zi) \prod_\zk (q^2-q^{-1} z_\zk)\\
 &\quad\times  \Phi_{\aj\bullet\hat{\ak},-p} (z_\zh,q^{\hat{\bk}},z_\zl)\\
 &= [s] \tau^{d(\bk)-d(\hat{\bk})}  \Phi_{\aj\bullet\hat{\ak},-p} (z_\zh,q^{\hat{\bk}},z_\zl).
\end{align*}
\end{proof}

Now, we know that $\Phi_{\ai,-p} (z_1,\ldots,z_{2n})$ lives at the vector space $\mathcal{V}_r$ if we ignore the outer variables, thus we can write it as linear combination of the $\Psi_\bk (z_{p+1},\ldots,z_{\hat{p}-1})$ and the coefficients will be a function on $\{z_1,\dots,z_p,z_{\hat{p}},\ldots,z_{\hat{1}}\}$.

Using Lemma~\ref{lema:prec}, we can construct these coefficients:
\begin{cor}\label{cor:prec}
Using the last lemma in the inner part we get:
\[
  \Phi_{\at}(z_\zh,q^\bk,z_\zl)=\tau^{d(\bk)} C_{\ak,\bk}(\tau) \Phi_\aj (z_\zh,z_\zl).
\]
\end{cor}

The proof is straightforward.

\bibliography{DN}
\bibliographystyle{amsplainhyper}
\end{document}